\documentclass{article}
\usepackage{amsmath,amsthm,amssymb,bbm}

\newcommand{\assign}{:=}
\newcommand{\tmop}[1]{\ensuremath{\operatorname{#1}}}
\newcommand{\tmstrong}[1]{\textbf{#1}}
\numberwithin{equation}{section}  
\newtheorem{corollary}{Corollary}

\newtheorem{theorem}{Theorem}
\newtheorem{lemma}{Lemma}

\newcommand{\XXint}[3]{{\setbox}0=\text{\ensuremath{#1 #2 #3 \int}}
{\vcenter{\text{\ensuremath{#2 #3}}}}{\kern}-.5{\tmwd}0}

\newcommand{\opn}[2]{\newcommand{\1}{\}} {\opn}{\Rm{Rm}} {\opn}{\Ric{Ric}}
{\opn}{\Rc{Rc}} {\opn}{\Scal{Sc}} {\opn}{\Tr{Tr}} {\opn}{\Trac{Tr}}
{\opn}detdet {\opn}{\diam{diam}} {\opn}{\dist{dist}} {\opn}{\Im}Im
{\opn}{\div}div {\opn}{\Ker{Ker}} {\opn}expexp {\opn}{\Vol{Vol}}
{\opn}{\exph{exph}} {\opn}{\Herm{Herm}} {\opn}{\End{End}} {\opn}{\Hess{Hess}}
{\opn}{\Vol{Vol}}}

\newcommand{\R}{\mathbb{R}}
\newcommand{\C}{\mathbb{C}}

\newcommand{\I}{\mathbb{I}}

\newcommand{\contract}{{\kern}-1.5pt{\vrule} width6.0pt height0.4pt depth0pt
{\vrule} width0.4pt height4.0pt depth0pt}
\newcommand{\retract}{{\kern}-1.5pt{\vrule} width0.4pt height4.0pt depth0pt
{\vrule} width6.0pt height0.4pt depth0pt}
\newcommand{\Openbox}{{\leavevmode} {\text{{\hfil}{\vrule}
width{\boxrulethickness} {\vbox} to{\Openboxwidth{{\advance}{\Openboxwidth}
-2{\boxrulethickness} {\hrule} height {\boxrulethickness}
width{\Openboxwidth}{\vfil} {\hrule} height{\boxrulethickness}}}{\vrule}
width{\boxrulethickness}{\hfil} }}}

\begin{document}

\title{Variation formulas for the complex components of the Bakry-Emery-Ricci endomorphism}
\author{\\
{\tmstrong{NEFTON PALI}}}\maketitle

\begin{abstract}
  We compute first variation formulas for the complex components of the
  Bakry-Emery-Ricci endomorphism along K\"ahler structures. Our formulas show
  that the principal parts of the variations are quite standard complex
  differential operators with particular symmetry properties on the complex
  decomposition of the variation of the K\"ahler metric. We show as application that
  the Soliton-K\"ahler-Ricci flow generated by the Soliton-Ricci flow
  represents a complex strictly parabolic system of the complex components of
  the variation of the K\"ahler metric.
\end{abstract}
\section{Introduction}
Let $(X,J,g)$ be a compact K\"ahler manifold and $\Omega>0$ be a smooth volume form. 
It is a quite basic fact that the $J$-invariant part of the $\Omega$-Bakry-Emery-Ricci tensor is 
the symmetric form associated to the Chern curvature of the metric induced by $\Omega$ over the anti-canonical bundle. 
We call this differential form the $\Omega$-Ricci form. 

Moreover the $J$-anti-invariant part of the 
$\Omega$-Bakry-Emery-Ricci tensor measures the default of the gradient of $\log \frac{dV_g}{\Omega}$ to be holomorphic.

In this paper we compute first variation formulas for the complex components of the endomorphism associated to the
  $\Omega$-Bakry-Emery-Ricci tensor along natural variations of K\"ahler structures considered in \cite{Pal1}. 
Clearly these complex components correspond to the endomorphisms associated to the $J$-invariant 
and $J$-anti-invariant parts of the $\Omega$-Bakry-Emery-Ricci tensor.

Our formulas show
  that the principal parts of the variations are quite standard complex
  differential operators with particular symmetry properties on the complex
  decomposition of the variation of the K\"ahler metric. (See theorem 1.)

Our formulas will be useful for the study of the pre-scattering problem in the Fano case (see \cite{Pal3}).
We do not expect that the scattering problem in \cite{Pal2} can be solved over general compact manifolds. 

In this paper we focus also in the particular case of variations of K\"ahler structures 
with $3$-symmetric covariant derivative of the variation of the metric. (See formulas (\ref{var-EndOm-RcCx-F}) and (\ref{var-II-cmpx-F}).)
We provide a relatively simple proof of the variation formulas in this case which is independent of the general ones.

The particular variation formulas allow to show a remarkable fact. Namely that the 
Soliton-K\"ahler-Ricci flow in \cite{Pal3} generated by the Soliton-Ricci flow in \cite{Pal2}
  represents a complex strictly parabolic system of the complex components of
  the variation of the K\"ahler metric. (See the evolution equations (\ref{evol-B}) and (\ref{evol-A}).)
  This result can be very usefull for proving uniform estimates along the Soliton-K\"ahler-Ricci flow.
It is known from \cite{Pal2} that the Soliton-Ricci flow represents a strictly parabolic equation.

We explain now the main steps of the proof of the general variation formulas for the complex 
components of the $\Omega$-Bakry-Emery-Ricci endomorphism.

We compute first the variation of the $\Omega$-Ricci form (See lemma \ref{Lm-var-O-Rc-fm}). A crucial point for the proof of 
this formula is is a special choice of complex geodesic coordinates. This choice is allowed by a diagonalizing property of a $U(n)$-action
over the space of symmetric complex matrices in \cite{Ho-Jo}.

The second step is the most difficult and technical part of the proof. It consist to show that the principal part of
variation of the symmetric form associated to the $\Omega$-Ricci form is the operator which appears in the variation of the
$\Omega$-Bakry-Emery-Ricci tensor in \cite{Pal1} acting on the $J$-anti-invariant part of the variation of the K\"ahler metric. 
(See corollary \ref{rm-vr-Rc-fm}.)

This step is also important because it allows to give a good expression of the variation of the $J$-anti-linear part of the 
$\Omega$-Bakry-Emery-Ricci endomorphism. Indeed in general this is not possible by a direct computation.

The third step consist to show some special complex Weitzenb\"ock type formulas (see lemma \ref{Cx-Weitz}) and some symmetry relations 
for standard complex operators.

The final step uses in a crucial way the properties of the variations of K\"ahler structures obtained in \cite{Pal1}. 
These properties are also of key importance for the proof of 
the particular variation formulas (\ref{var-EndOm-RcCx-F}) and (\ref{var-II-cmpx-F}).
\section{The complex components of the Bakry-Emery-Ricci endomorphism}

Let $\Omega > 0$ be a smooth volume form over an oriented Riemannian manifold
$(X, g)$. We define the $\Omega$-Bakry-Emery-Ricci tensor of $g$ as
$$
\tmop{Ric}_g (\Omega) \;\;: = \;\; \tmop{Ric} (g)
   \;\, + \;\, \nabla_g \,d \log \frac{dV_g}{\Omega} \;. 
$$
We define the $\Omega$-divergence operator of a tensor $\alpha$ as
\begin{eqnarray*}
  \tmop{div}^{^{_{_{\Omega}}}}_g \alpha \;\; \assign \;\; e^f \tmop{div}_g \left(
  e^{- f} \alpha \right) \;\;=\;\; \tmop{div}_g \alpha \;\,-\;\, \nabla_g \,f \;\neg\; \alpha\;,
\end{eqnarray*}
with $f \assign \log \frac{d V_g}{\Omega}$. With this notation the first
variation of the $\Omega$-Bakry-Emery-Ricci tensor (see \cite{Pal1}) is given by the
formula
\begin{equation}
  \label{var-Om-Ric} 2\, \frac{d}{d t} \tmop{Ric}_{g_t} (\Omega) \;\;=\;\;
  \tmop{div}^{^{_{_{\Omega}}}}_{g_t}\mathcal{D}_{g_t}  \dot{g}_t\;,
\end{equation}
where $\mathcal{D}_g \assign \hat{\nabla}_g \;-\; 2\, \nabla_g$, with
$\hat{\nabla}_g$ being the symmetrization of
$\nabla_g$ acting on symmetric 2-tensors.
Explicitly
\begin{eqnarray*}
  \hat{\nabla}_g \,\alpha \,(\xi_0, ..., \xi_p) \;\; : = \;\; \sum_{j = 0}^p \nabla_g\,
  \alpha \,(\xi_j, \xi_0, ..., \hat{\xi}_j, ..., \xi_p)\;,
\end{eqnarray*}
for all $p$-tensors $\alpha$. Let now $(X, J)$ be a complex manifold. Then the
volume form $\Omega > 0$ induces a hermitian metric $h_{\Omega}$ over the
canonical bundle $K_{_{X, J}} \assign \Lambda_{_J}^{n, 0} T^{\ast}_{_X} $
given by the formula
$$
 h_{\Omega} (\alpha, \beta) \;\;\assign\;\; \frac{n! \,i^{n^2} \alpha \wedge
   \bar{\beta} }{\Omega} \;. 
$$
By abuse of notations we will denote by $\Omega^{- 1}$ the metric $h_{\Omega}$. 
The dual metric $h_{\Omega}^{\ast}$ on the anti-canonical bundle 
$K^{-1}_{_{X, J}} = \Lambda_{_J}^{n, 0} T_{_X}$ is given by the formula
$$
 h_{\Omega}^{\ast} (\xi, \eta) \;\;=\;\; (- i)^{n^2} \Omega \left( \xi,
   \bar{\eta} \right) / n! \;. 
$$
Abusing notations again, we denote by $\Omega$ the dual metric
$h_{\Omega}^{\ast}$. We define the $\Omega$-Ricci form
$$
\tmop{Ric}_{_J} \left( \Omega \right) \;\;\assign\;\; i\,\mathcal{C}_{\Omega}  \big(
   K^{- 1}_{_{X, J}} \big) \;\;=\;\; - \;\,i\,\mathcal{C}_{\Omega^{- 1}} \big( K_{_{X,
   J}} \big)\;, 
$$
where $\mathcal{C}_h (L)$ denotes the Chern curvature of a hermitian line
bundle. In particular we observe the identity $\tmop{Ric}_{_J} (\omega) =
\tmop{Ric}_{_J} (\omega^n)$. We remind also that for any $J$-invariant
K\"ahler metric $g$ the associated symplectic form $\omega \assign g J$
satisfies the elementary identity
\begin{eqnarray*}
  \tmop{Ric} (g) \;\; = \;\; -\;\, \tmop{Ric}_{_J} (\omega) J\;.
\end{eqnarray*}
Moreover for all twice differentiable function $f$ hold the identity
\begin{equation}
  \label{cx-dec-Hess} \nabla_g \,d\,f \;\;=\;\; - \;\, \left(  i\, \partial_{_J}
  \bar{\partial}_{_J} \,f\right)  J \;\,+\;\, g\, \bar{\partial}_{_{T_{X, J}}}
  \nabla_g \,f\; .
\end{equation}
We infer the decomposition identity
\begin{equation}
  \label{cx-dec-Ric} \tmop{Ric}_g (\Omega) \;\;=\;\; -\;\, \tmop{Ric}_{_J} (\Omega) J \;\,+\;\, g\,
  \bar{\partial}_{_{T_{X, J}}} \nabla_g \log \frac{dV_g}{\Omega} \;.
\end{equation}
Let $v \in S_{_{\mathbbm{R}}}^2 T^{\ast}_X,$ $\alpha \in
\Lambda_{_{\mathbbm{R}}}^2 T^{\ast}_X .$ We define $v^{\ast}_g \assign g^{- 1}
v$ and $\alpha_g^{\ast} \assign \omega^{- 1} \alpha$. With this notations the
decomposition formula (\ref{cx-dec-Ric}) implies
\begin{equation}
  \label{dec-end-Ric} \tmop{Ric}^{\ast}_{g_{_{}}} (\Omega) \;\;=\;\;
  \tmop{Ric}^{\ast}_{_J} (\Omega)_g \;\,+\;\, \bar{\partial}_{_{T_{X, J}}}
  \nabla_g \log \frac{d V_g}{\Omega} \;.
\end{equation}
The sections that will follow are devoted to the study of the first variation
of the two complex components in (\ref{cx-dec-Ric}) and (\ref{dec-end-Ric}).

\section{The first variation of the $\Omega$-Ricci form}

Let $\mathcal{M} \subset C^{\infty} (X, S^2_{_{\mathbbm{R}}} T_X^{\ast})$ be
the space of smooth Riemannian metrics over a compact manifold $X$, let
$\mathcal{J} \subset C^{\infty} (X, \tmop{End}_{_{\mathbbm{R}}} (T_X))$ be the
set of smooth almost complex structures and let
$$
\mathcal{KS} \;\; : = \;\; \Big\{ (J, g) \in
   \mathcal{J} \times \mathcal{M} \hspace{0.25em} \mid \hspace{0.25em}
   \hspace{0.25em} g \;\, = \;\, J^{\ast} g\,J
   \;, \;\;  \nabla_g \,J \;\,
   = \;\, 0 \;\Big\} \;, 
$$
be the space of K\"ahler structures. We remind that if $A \in
\tmop{End}_{_{\mathbbm{R}}} (T_X)$ then its transposed $A^T_g$ with respect to
$g$ is given by $A^T_g = g^{- 1} A^{\ast} g$. We observe that the
compatibility condition $g = J^{\ast} gJ$ is equivalent to the condition
$J^T_g = - J$. We show now an elementary formula.

\begin{lemma}
  Let $(g_t, J_t)_t \subset \mathcal{K}\mathcal{S}
  \text{}$ be a smooth path and let $\xi$ be a smooth vector field. Then hold
  the first variation identity for the $\bar{\partial}$-operator acting
  on vector fields
  \begin{equation}
    \label{var-dbar-vflds} 2 \left( \frac{d}{dt} \hspace{0.25em}
    \bar{\partial}_{_{T_{X, J_t}}} \right) \xi 
\;\;=\;\; -\;\, \xi \; \neg \; J_t
    \nabla_{g_t} \dot{J}_t  \;\,+\;\, J_t \nabla_{g_t}
    \xi \hspace{0.25em}  \dot{J}_t \;\,+\;\, \dot{J}_t \nabla_{g_t} \xi
    \hspace{0.25em} J_t \;.
  \end{equation}
\end{lemma}

\begin{proof}
  The fact that in the K\"ahler case the Chern connection coincides with the
  Levi-Civita connection implies 
$$
2 \,
  \bar{\partial}_{_{T_{X, J_t}}} \xi \;\;=\;\; \nabla_{g_t} \xi \;\,+\;\, J_t
  \nabla_{g_t} \xi \hspace{0.25em} J_t\;.
$$
Time deriving this identity we infer
  \begin{equation}
    \label{vr-dbar-Vfld} 2 \left( \frac{d}{dt} \hspace{0.25em}
    \bar{\partial}_{_{T_{X, J_t}}} \right) \xi 
\;\;=\;\;
\dot{\nabla}_{g_t} \xi \;\,+\;\,
    \dot{J}_t \nabla_{g_t} \xi \hspace{0.25em} J_t 
\;\,+\;\,
J_t \dot{\nabla}_{g_t} \xi \hspace{0.25em} J_t
    \;\,+\;\, J_t \nabla_{g_t} \xi \hspace{0.25em}
    \dot{J}_t \; .
  \end{equation}
  On the other hand time deriving the K\"ahler condition $\nabla_{g_t} J_t =
  0$ we get the identity
$$
\dot{\nabla}_{g_t} J_t \;\,+\;\, \nabla_{g_t}  \dot{J}_t \;\;=\;\; 0\;, 
$$
which writes explicitly as
  \begin{equation}
    \label{der-Kah-cnd}  \dot{\nabla}_{g_t} (\xi, J_t \eta) \;\, -\;\,
     J_t \dot{\nabla}_{g_t} (\xi, \eta) \;\, +\;\,
     \nabla_{g_t}  \dot{J}_t (\xi, \eta) \;\; =\;\;0\; .
  \end{equation}
We remind now that the tensor $\dot{\nabla}_{g_t}$ is symmetric (see \cite{Bes} or
  the identity (4) in \cite{Pal1}). Thus the identity (\ref{der-Kah-cnd})
  rewrites as
$$
 \dot{\nabla}_{g_t} (J_t \eta, \xi) \;\, - \;\, J_t
     \dot{\nabla}_{g_t} (\eta, \xi) \;\,+ \;\,
     \nabla_{g_t} \dot{J}_t (\xi, \eta) \;\;=\;\; 0
     \hspace{0.25em}, 
$$
  which multiplied by $J_t$, is equivalent to the identity
$$
 \dot{\nabla}_{g_t} (\eta, \xi) \;\, + \;\, J_t
     \dot{\nabla}_{g_t} (J_t \eta, \xi) \;\, + \;\, J_t
     \nabla_{g_t}  \dot{J}_t (\xi, \eta) \;\; = \;\; 0\;, 
$$
  i.e to the identity
$$
\dot{\nabla}_{g_t} \xi \;\,+\;\, J_t \dot{\nabla}_{g_t} \xi \hspace{0.25em}
  J_t \;\;=\;\; - \;\, \xi \;\neg\; J_t \nabla_{g_t}  \dot{J}_t \;.
$$
Plunging this
  in to the identity (\ref{vr-dbar-Vfld}) we infer the variation formula
  (\ref{var-dbar-vflds}).
\end{proof}

We introduce now a few useful notations. For any section $S \in C^{\infty}
(X, \tmop{End}_{_{\R}} (T_{_X}))$ and any $\xi, \eta \in T_X$ we define the
complex operators
\begin{eqnarray*}
  \nabla_{g, J}^{1, 0} \,S\, (\xi, \eta) & : = &  \frac{1}{2} \, \Big[ \nabla_g\, S\,
  (\xi, \eta) \;\,- \;\, J\, \nabla_g \,S\, (J \xi, \eta)
  \Big] \;,\\
  &  & \\
  \nabla_{g, J}^{0, 1} \,S\, (\xi, \eta) & : = &  \frac{1}{2} \, \Big[ \nabla_g \,S\,
  (\xi, \eta) \;\,+ \;\, J\, \nabla_g \,S\, (J \xi, \eta)
  \Big] \; .
\end{eqnarray*}
We define also the $J$-anti-linear operator 
$$
\nabla_{g, J}^{0, 1} \,S \cdot
 \eta \;\;: =\;\; \nabla_{g, J}^{0, 1} \,S\, (\cdot, \eta)\;.
$$
We show now the following first variation formula for the
$\Omega$-Ricci form.
\begin{lemma}
  \label{Lm-var-O-Rc-fm}Let $(g_t, J_t)_t \subset
  \mathcal{K}\mathcal{S}$ be a smooth path such that $\dot{J}_t = (
  \dot{J}_t)_{g_t}^T$ and let $\Omega > 0$ be a smooth volume form over $X$. Then hold the first
  variation formula
  \begin{equation}
    \label{vr-O-Rc-fm} 2 \,\frac{d}{d t} \tmop{Ric}_{_{J_t}} (\Omega) \;\;=\;\; d  \tmop{div}^{^{_{_{\Omega}}}}_{g_t} \big(g_t \dot{J_t}\big)\; .
  \end{equation}
\end{lemma}

\begin{proof}
  We remind first a general identity. Let $(L, \bar{\partial}_L,
  h)$ be a hermitian holomorphic line bundle over a complex manifold $(X, J)$
  and let $D_{L, h} \,=\, \partial_{L, h} \,+\, \bar{\partial}_L$ be the induced
  Chern connection. We observe that for any local non-vanishing section
  $\sigma \in C^{\infty} (U, L \smallsetminus 0)$ over an open set $U\subset
X$, hold the identity
  \begin{eqnarray*}
    \sigma^{- 1} \,\partial_{L, h} \sigma (\eta) & = &  | \sigma |_h^{- 2}\,
    h \big( \partial_{L, h} \sigma (\eta), \sigma \big)\\
    &  & \\
    & = & | \sigma |_h^{- 2} \Big[ \eta_{_J}^{1, 0} . \,| \sigma |_h^2 \;\,-\;\, h \big(
    \sigma, \bar{\partial}_L \sigma (\eta) \big)\Big]\\
    &  & \\
    & = & \eta_{_J}^{1, 0} . \,\log | \sigma |_h^2 \;\,-\;\, \overline{\sigma^{- 1}\,
    \bar{\partial}_L \sigma (\eta)}\;,
  \end{eqnarray*}
  for all $\eta \in T_X$. We infer the formula
  \begin{eqnarray*}
    i \,\sigma^{- 1} D_{L, h} \sigma (\eta) & = &  i\, \eta_{_J}^{1, 0} .\,
    \log | \sigma |_h^2 \;\,+\;\, 2\, \Re e \big( i\, \sigma^{- 1} \,\bar{\partial}_L
    \sigma (\eta) \big) \;.
  \end{eqnarray*}
  In the case $L = K_{_{X, J_t}}^{- 1}$ and $h \equiv \Omega$ we get for all
$$
 \xi_t \;\;=\;\; \xi^{1, 0}_{1, t} \wedge \ldots \wedge \xi^{1, 0}_{n, t} \;\;\in\;\;
     C^{\infty} \left( U, K_{_{X, J_t}}^{- 1} \smallsetminus 0 \right)\;,\quad \xi_j
     \;\;\in\;\; C^{\infty} \left( U, T_X \right)\;, 
$$
  and all $\eta \in C^{\infty} \left( U, T_X \right)$ the formula for the
  $1$-form $\alpha_t$
  \begin{eqnarray*}
    \alpha_t (\eta) &\assign& i\, \xi_t^{- 1} D_{K_{_{X, J_t}}^{- 1}, \Omega}
    \,\xi_t (\eta) 
\\
\\
& = &  i\, \eta_{_{J_t}}^{1, 0} . \,\log \left[ \Omega
    \left( \xi_t, \bar{\xi}_t \right) / i^{n^2} \right] 
\;\,+\;\, 2\, \Re e \Big( i\,
    \xi_t^{- 1}\, \bar{\partial}_{K_{_{X, J_t}}^{- 1}, \Omega} \,\xi_t \,(\eta)
    \Big) \;.
  \end{eqnarray*}
  We remind also the local expression
  \begin{eqnarray*}
    \tmop{Ric}_{_{J_t}} (\Omega) \;\; = \;\; i\, D^2_{K_{_{X, J_t}}^{- 1}, \Omega} \;\;=\;\; d
    \alpha_t\;.
  \end{eqnarray*}
Let now $f_t \assign \log \frac{d V_{g_t}}{\Omega}$.  We fix an arbitrary space time point $(x_0, t_0)$ and we choose arbitrary
  $g_{t_0}$-geodesic and $J_{t_0}$-holomorphic coordinates $(z_1, \ldots,
  z_n)$ centered at the point $x_0$ such that $\xi_k =
  \frac{\partial}{\partial x_k} .$ A particular choice of a $J_{t_0}
  (x_0)$-complex basis $\left( \xi_k (x_0) \right)_k$ which is also
  $\omega_{t_0} (x_0)$-orthonormal (with $\omega_{t_0} \assign g_{t_0}
  J_{t_0}$) will be made at the end of the proof. We expand now the time
  derivative $\dot{\alpha}_t .$ \ Indeed at the time $t_0$ hold the equalities
  \begin{eqnarray*}
    \dot{\alpha}_t (\eta) & = & \frac{1}{2} \, (\dot{J}_t \eta)\, .\, \log \left[
    \Omega \left( \xi_t, \bar{\xi}_t \right) / i^{n^2} \right] \\
    &  & \\
    & + & \frac{1}{2} \,\eta_{_{J_t}}^{1, 0} \,.\, \sum_{l = 1}^n  \;\frac{\Omega
    \left( \xi^{1, 0}_{1, t} \wedge \ldots \wedge \left( \dot{J}_t \xi_l
    \right)_{J_t}^{1, 0} \wedge \ldots \wedge \xi^{1, 0}_{n, t},
    \bar{\xi}_t \right) }{\Omega (\xi_t, \bar{\xi}_t)}\\
    &  & \\
    & - & \frac{1}{2} \,\eta_{_{J_t}}^{1, 0} \,.\, \sum_{l = 1}^n  \;\frac{\Omega
    \left( \xi_t, \xi^{0, 1}_{1, t} \wedge \ldots \wedge \left( \dot{J}_t
    \xi_l \right)_{J_t}^{0, 1} \wedge \ldots \wedge \xi^{0, 1}_{n, t}  \right)
    }{\Omega (\xi_t, \bar{\xi}_t)}\\
    &  & \\
    & + & 2 \,\Re e \left[  \frac{d}{d t} \left( i\, \xi_t^{- 1}\,
    \bar{\partial}_{K_{_{X, J_t}}^{- 1}, \Omega} \,\xi_t \,(\eta) \right)
    \right] \\
    &  & \\
    & = & - \;\,\frac{1}{2} \,d\, f_t \cdot \dot{J}_t \eta \;\,+\;\, \frac{1}{2}\,  (\dot{J}_t
    \eta)\, .\, \log \left[ d V_{g_t}  \left( \xi_t, \bar{\xi}_t \right) /
    i^{n^2} \right] \\
    &  & \\
    & + & \frac{1}{2} \,\eta_{_{J_t}}^{1, 0} \,.\, \left( \xi^{- 1}_t \cdot \sum_{l
    = 1}^n \;\xi^{1, 0}_{1, t} \wedge \ldots \wedge \left( \dot{J}_t \xi_l
    \right)_{J_t}^{1, 0} \wedge \ldots \wedge \xi^{1, 0}_{n, t}  \right)\\
    &  & \\
    & - & \frac{1}{2} \,\eta_{_{J_t}}^{1, 0} \,.\, \left( \bar{\xi}_t^{- 1}
    \cdot \sum_{l = 1}^n \;\xi^{0, 1}_{1, t} \wedge \ldots \wedge \left(
    \dot{J}_t \xi_l \right)_{J_t}^{0, 1} \wedge \ldots \wedge \xi^{0, 1}_{n,
    t} \right)\\
    &  & \\
    & + & \Re e \left[ \xi^{- 1}_t \cdot \sum_{l = 1}^n \xi^{1, 0}_{1, t}
    \wedge \ldots \wedge T \wedge \ldots \wedge \xi^{1, 0}_{n, t}  \right],
  \end{eqnarray*}
with 
$$
T\;\;:=\;\;\left( \nabla_{g_t}  \dot{J}_t (\xi_l, \eta) \;\,-\;\,
    \nabla_{g_t} \xi_l ( \dot{J}_t \eta) \;\,+\;\, \dot{J}_t \nabla_{g_t} \xi_l (\eta)
    \right)_{J_t}^{1, 0}\;,
$$
  thanks to the fact that $\bar{\partial}_{_{T_{X, J_t}}} \xi_l \equiv 0$
  at time $t_0$ and thanks to the variation formula (\ref{var-dbar-vflds}). 
We define now the real $1$-form $\tmop{Tr}_{_{\mathbbm{R}}} \nabla_{g_t} 
\dot{J_t}$ by the formula
$$
 \left( \tmop{Tr}_{_{\mathbbm{R}}} \nabla_{g_t}  \dot{J_t} \right) (\xi)
   \;\;\assign\;\; \tmop{Tr}_{_{\mathbbm{R}}} \nabla_{g_t}  \dot{J_t} (\cdot,
   \xi)\;,\quad \forall\;\,\xi \;\;\in\;\; T_X\; . 
$$
Expanding further the last expression of $\dot{\alpha}_t (\eta)$, simplifying and using the identity
  \begin{eqnarray*}
    \tmop{Tr}_{_{\mathbbm{R}}} A \;\; = \;\; 2\, \Re e \left(
    \tmop{Tr}_{_{\mathbbm{C}}} A_{_J}^{1, 0} \right)\;,\quad \forall \;\,A \;\;\in\;\;
    \tmop{End}_{_{\mathbbm{R}}} (T_X)\;,
  \end{eqnarray*}
we infer the following expression
  \begin{eqnarray*}
    \dot{\alpha}_t (\eta) & = & - \;\,\frac{1}{2} \,d\, f_t \cdot \dot{J}_t \eta\\
    &  & \\
    & + &  \frac{1}{2} \,\eta_{_{J_t}}^{1, 0} \,. \left( \xi^{- 1}_t \cdot
    \sum_{l = 1}^n \xi^{1, 0}_{1, t} \wedge \ldots \wedge \left( \dot{J}_t
    \xi_l \right)_{J_t}^{1, 0} \wedge \ldots \wedge \xi^{1, 0}_{n, t} 
    \right)\\
    &  & \\
    & - & \frac{1}{2}  \,\bar{\xi}_t^{- 1} \cdot \sum_{l = 1}^n \;\xi^{0,
    1}_{1, t} \wedge \ldots \wedge \left( \nabla_{g_t, J_t}^{0, 1}  \dot{J}_t
    (\eta, \xi_l) \right)_{J_t}^{0, 1} \wedge \ldots \wedge \xi^{0, 1}_{n,
    t}\\
    &  & \\
    & + & \frac{1}{2}  \left( \tmop{Tr}_{_{\mathbbm{R}}} \nabla_{g_t} 
    \dot{J_t} \right) (\eta)\\
    &  & \\
    & + & \Re e \left[ \xi^{- 1}_t \cdot \sum_{l = 1}^n \;\xi^{1, 0}_{1, t}
    \wedge \ldots \wedge H \wedge
    \ldots \wedge \xi^{1, 0}_{n, t}  \right] \;,
  \end{eqnarray*}
with
$$
H\;\;:=\;\;\left( \nabla_{g_t, J_t}^{0, 1}  \dot{J}_t (\xi_l,
    \eta) + \dot{J}_t \nabla_{g_t} \xi_l (\eta) \right)_{J_t}^{1, 0}\;.
$$
  We rearrange the previous expression of $\dot{\alpha}_t (\eta)$ by means of the symmetry
  \begin{equation}
    \label{Tg-cx-str} \nabla_{g_t, J_t}^{0, 1}  \dot{J}_t \;\;\in\;\;
    S_{_{\mathbbm{R}}}^2 T^{\ast}_X \otimes T_X\;,
  \end{equation}
  (see lemma 7 in \cite{Pal1}) in order to get at the time $t_0$ the
  identities,
  \begin{eqnarray*}
    \dot{\alpha}_t (\eta) & = & \frac{1}{2}  \left( \tmop{Tr}_{_{\mathbbm{R}}}
    \nabla_{g_t}  \dot{J_t} \;\,-\;\, d\, f_t \cdot \dot{J}_t \right) (\eta)  \\
    &  & \\
    & + & \frac{1}{2} \,\eta_{_{J_t}}^{1, 0} \,.\, \left( \xi^{- 1}_t \cdot \sum_{l
    = 1}^n \;\xi^{1, 0}_{1, t} \wedge \ldots \wedge \left( \dot{J}_t \xi_l
    \right)_{J_t}^{1, 0} \wedge \ldots \wedge \xi^{1, 0}_{n, t}  \right)\\
    &  & \\
    & + & \frac{1}{2} \,\xi^{- 1}_t \cdot \sum_{l = 1}^n \;\xi^{1, 0}_{1, t}
    \wedge \ldots \wedge \left( \nabla_{g_t, J_t}^{0, 1}  \dot{J}_t (\eta,
    \xi_l) \right)_{J_t}^{1, 0} \wedge \ldots \wedge \xi^{1, 0}_{n, t}\\
    &  & \\
    & + & \Re e \left[ \xi^{- 1}_t \cdot \sum_{l = 1}^n \;\xi^{1, 0}_{1, t}
    \wedge \ldots \wedge \left( \dot{J}_t \nabla_{g_t} \xi_l (\eta)
    \right)_{J_t}^{1, 0} \wedge \ldots \wedge \xi^{1, 0}_{n, t}  \right]\\
    &  & \\
    & = & 
\frac{1}{2}  \left( \tmop{Tr}_{_{\mathbbm{R}}}
    \nabla_{g_t}  \dot{J_t} \;\,-\;\, d\, f_t \cdot \dot{J}_t \right) (\eta) \\
    &  & \\
    & + & \eta \,. \left( \xi^{- 1}_t \cdot \sum_{l = 1}^n \;\xi^{1, 0}_{1, t}
    \wedge \ldots \wedge \left( \dot{J}_t \xi_l \right)_{J_t}^{1, 0} \wedge
    \ldots \wedge \xi^{1, 0}_{n, t}  \right) \\
    &  & \\
    & + & \Re e \left[ \xi^{- 1}_t \cdot \sum_{l = 1}^n \;\xi^{1, 0}_{1, t}
    \wedge \ldots \wedge \left( \dot{J}_t \nabla_{g_t} \xi_l (\eta)
    \right)_{J_t}^{1, 0} \wedge \ldots \wedge \xi^{1, 0}_{n, t}  \right] \;.
  \end{eqnarray*}
  We take now two vector fields $\eta, \mu \in C^{\infty} (U, T_X)$ with
  constant coefficients with respect to the coordinates $(z_1, \ldots, z_n)$.
  Then at the space time point $(x_0, t_0)$ hold the identity
  \begin{eqnarray*}
    &&2 \left( \frac{d}{d t} \tmop{Ric}_{_{J_t}} (\Omega) \right) (\eta, \mu) 
\\
\\
&
    = &  \left[ d \left( \tmop{Tr}_{_{\mathbbm{R}}} \nabla_{g_t}  \dot{J}_t
    \;\, - \;\, d\, f_t \cdot \dot{J}_t\right)\right] (\eta, \mu)\\
    &  & \\
    & + & 2\, \Re e \left[ \xi^{- 1}_t \cdot \sum_{l = 1}^n \;\xi^{1, 0}_{1, t}
    \wedge \ldots \wedge \left( \dot{J}_t \nabla_{g_t, \eta} \nabla_{g_t, \mu}\,
    \xi_l \right)_{J_t}^{1, 0} \wedge \ldots \wedge \xi^{1, 0}_{n, t} 
    \right]\\
    &  & \\
    & - & 2\, \Re e \left[ \xi^{- 1}_t \cdot \sum_{l = 1}^n \;\xi^{1, 0}_{1, t}
    \wedge \ldots \wedge \left( \dot{J}_t \nabla_{g_t, \mu} \nabla_{g_t, \eta}\,
    \xi_l \right)_{J_t}^{1, 0} \wedge \ldots \wedge \xi^{1, 0}_{n, t}  \right]
    \\
    &  & \\
    & = &  \left[ d \left( \tmop{Tr}_{_{\mathbbm{R}}} \nabla_{g_t}  \dot{J}_t
    \;\, - \;\, d\, f_t \cdot \dot{J}_t\right)\right] (\eta, \mu)\\
    &  & \\
    & + & 2 \,\Re e \left[ \xi^{- 1}_t \cdot \sum_{l = 1}^n \;\xi^{1, 0}_{1, t}
    \wedge \ldots \wedge \left( \dot{J}_t \mathcal{R}_{g_t} (\eta, \mu) \,\xi_l
    \right)_{J_t}^{1, 0} \wedge \ldots \wedge \xi^{1, 0}_{n, t}  \right] \;.
  \end{eqnarray*}
  For notation simplicity we define 
$$
\zeta_k \;\;\assign\;\; \xi^{1, 0}_{k, t_0} (x_0)
  \;\;=\;\; \frac{\partial}{\partial z_k} _{\mid_{x_0}}\;.
$$
With this notation hold the
  local expression
  \begin{eqnarray*}
    \dot{J}_{t_0} (x_0) \;\; = \;\; C_{k, \bar{l}}  \;\bar{\zeta}_l^{\ast} \otimes
    \zeta_k \;\,+\;\, \overline{C}_{k, \bar{l}} \;\zeta^{\ast}_l \otimes \bar{\zeta}_k\;.
  \end{eqnarray*}
A change of $J_{t_0} (x_0)$-complex basis $\left( \xi_k (x_0)
  \right)_k$ which is also $\omega_{t_0} (x_0)$-orthonormal will change the
  symmetric complex matrix $C = (C_{k, \bar{l}})$ in to a matrix of type $D
  \assign U^T C U$, with $U \in U (n)$. 

A standard linear algebra result
  (see corollary 4.4.4, pp 204-205 in Ho-Jo) shows that under this action of
  the unitary group the matrix $D$ can be reduced to a real and diagonal one.
  Therefore we can choose the $\omega_{t_0} (x_0)$-orthonormal 
$J_{t_0} (x_0)$-complex basis $\left( \xi_k (x_0) \right)_k$ such that
  \begin{eqnarray*}
    \dot{J}_{t_0} (x_0) \;\; = \;\; C_l\;  \bar{\zeta}_l^{\ast} \otimes \zeta_l \;\,+\;\, C_l\;
    \zeta^{\ast}_l \otimes \bar{\zeta}_l\;,
  \end{eqnarray*}
  with $C_l \in \mathbbm{R}$. We write $\eta \;=\; \eta_k\, \zeta_k \;+\; \bar{\eta}_k\, 
  \bar{\zeta}_k$ and $\mu \;=\; \mu_k\, \zeta_k\; +\; \bar{\mu}_k\,  \bar{\zeta}_k$ at the
  point $x_0$. Moreover let $\varphi$ be a real valued smooth function in a
  neighborhood of $x_0$ such that 
$$
\omega_{t_0} \;\;=\;\; \frac{i}{2}\,
  \partial_{_{J_0}} \bar{\partial}_{_{J_0}} \varphi\;.
$$
Then at the space
  time point $(x_0, t_0)$ hold the local expression
  \begin{eqnarray*}
    \mathcal{R}_{g_t} (\eta, \mu) \;\; = \;\; -\;\, \varphi_{j, \bar{k}, l, \bar{h}}\,
    \big(\eta_j  \,\bar{\mu}_k \;\,-\;\, \mu_j  \,\bar{\eta}_k\big) \;\zeta^{\ast}_l \otimes \zeta_h
    \;\,+\;\, \tmop{Conjuguate}\;.
  \end{eqnarray*}
We infer the identity
  \begin{eqnarray*}
&& 2\, \Re e \left[ \xi^{- 1}_t \cdot \sum_{l = 1}^n \;\xi^{1, 0}_{1, t} \wedge
    \ldots \wedge \left( \dot{J}_t \mathcal{R}_{g_t} (\eta, \mu) \xi_l
    \right)_{J_t}^{1, 0} \wedge \ldots \wedge \xi^{1, 0}_{n, t} \right]  
\\
\\
& = &
    2 \,\Re e \Big[ C_l \,\varphi_{k, \bar{j}, l, \bar{l}} \,\big(\eta_k  \,\bar{\mu}_j \;\,-\;\,
    \mu_k  \,\bar{\eta}_j\big) \Big]\\
    &  & \\
    & = & 0\;,
  \end{eqnarray*}
  at the space time point $(x_0, t_0)$. We deduce the variation identity
\begin{equation}
    \label{dec-vr-O-Rc} 2 \, \frac{d}{dt} \tmop{Ric}_{_{J_t}}
    (\Omega) = d \left( \tmop{Tr}_{_{\mathbbm{R}}} \nabla_{g_t}  \dot{J}_t
    \;\, - \;\, \dot{J}^{\ast}_t d \,f_t 
    \right),
  \end{equation}
where $\dot{J}^{\ast}_t d \,f_t \;\assign\; d\, f_t \cdot \dot{J}_t$.
Finally using the symmetry identities $\dot{J}_t\;=\;(\dot{J}_t)^T_{g_t}$ 
and $\nabla_{g_t, \xi} \, \dot{J}_t\;=\;(\nabla_{g_t, \xi} \, \dot{J}_t)^T_{g_t}$
we obtain
\begin{eqnarray*}
2 \, \frac{d}{dt} \tmop{Ric}_{_{J_t}}(\Omega)
\;\; = \;\;
d\left[\tmop{div}_{g_t} (g_t \dot{J}_t) \;\,-\;\, \nabla_{g_t}\,f_t\;\neg\;(g_t \dot{J}_t)\right] \;,
\end{eqnarray*}
which shows the required conclusion.
\end{proof}


For any $h \in C^{\infty} (X, (T^{\ast}_X)^{\otimes 2})$ we define
respectively its $J$-invariant part and its $J$-anti-invariant part as
\begin{eqnarray*}
  h'_{_J} & = & \frac{1}{2}  \left( h \;\,+\;\, J^{\ast} h\, J \right)\;,\\
  &  & \\
  h''_{_J} & = & \frac{1}{2}  \left( h \;\,-\;\, J^{\ast} h\, J \right)\;,
\end{eqnarray*}
and we observe the trivial identities 
$h'_{_J} \;=\; g \left( h^{\ast}_g
\right)_{_J}^{1, 0}$, 
$h''_{_J} \;=\; g \left( h^{\ast}_g \right)_{_J}^{0,
1}$, where as usual $h^{\ast}_g \;\assign\; g^{- 1} h$.

With this notations hold the following lemma.

\begin{corollary}
  \label{rm-vr-Rc-fm} Under the the assumptions of lemma \ref{Lm-var-O-Rc-fm}
  hold the identity
  \begin{eqnarray*}
    2 \,\frac{d}{d t}  \left( \tmop{Ric}_{_{J_t}} (\Omega)  J_t \right)  &
    = & - \;\,\tmop{div}^{^{_{_{\Omega}}}}_{g_t} \mathcal{D}_{g_t}  \dot{g}''_t \\
    &  & \\
    & - & g_t \left( J_t \nabla_{g_t} f_t \;\neg\; \nabla_{g_t}  \dot{J}_t \;\,+\;\,
    \dot{J}_t \nabla^2_{g_t} f_t \,J_t \;\,-\;\, J_t \nabla^2_{g_t} f_t  \,\dot{J}_t
    \right) \;.
  \end{eqnarray*}
  where $ \dot{g}''_t \assign ( \dot{g}_t)_{_{J_t}}''$ and $f_t \assign \log \frac{d V_{g_t}}{\Omega}$.
\end{corollary}

\begin{proof}
  We consider first the elementary decomposition
  \begin{eqnarray*}
    \tmop{div}^{^{_{_{\Omega}}}}_{g_t} \mathcal{D}_{g_t}  \dot{g}''_t  \;\; = \;\;
    \tmop{div}_{g_t} \mathcal{D}_{g_t}  \dot{g}''_t \;\, -\;\,
     \nabla_{g_t} f_t \;\neg\; \mathcal{D}_{g_t}  \dot{g}''_t\;,
  \end{eqnarray*}
  and we will show the identity
  \begin{equation}
    \label{dec-O-RcI} \tmop{div}_{g_t}  \mathcal{D}_{g_t}  \dot{g}''_t \;\;=\;\; -\;\,
    \left( d \tmop{Tr}_{_{\mathbbm{R}}} \nabla_{g_t}  \dot{J}_t \right) 
    J_t \;\,-\;\, 2\, \tmop{Ric}_{_{J_t}} (\omega_t)\, \dot{J}_t \;.
  \end{equation}
  In a second time we will show the identity
\begin{eqnarray}
\label{dec-O-Rc2}    
- \;\nabla_{g_t} f_t \;\neg\; \mathcal{D}_{g_t}  \dot{g}_t & = & \left( d\,
    \dot{J}^{\ast}_t d \,f_t \right) J_t \;-\; 2\, \left( i\, \partial_{_{J_t}}
    \bar{\partial}_{_{J_t}} f_t \right) \dot{J}_t\nonumber
\\\nonumber
    &  & \\
    & - & g_t \left( J_t \nabla_{g_t} f_t \;\neg\; \nabla_{g_t}  \dot{J}_t \;+\;
    \dot{J}_t \nabla^2_{g_t} f_t \,J_t \;-\; J_t \nabla^2_{g_t} f_t\,  \dot{J}_t
    \right).\;\qquad
  \end{eqnarray}
Then the
  conclusion will follow from the variation formula (\ref{dec-vr-O-Rc}) 
obtained in the proof of lemma \ref{Lm-var-O-Rc-fm}. We divide now the proof
  in a few steps.

  {\tmstrong{(A)}} In this step we explicit in local coordinates the r.h.s of
  the identity (\ref{dec-vr-O-Rc}). We will keep the notations for the
  geodesic coordinates used at the end the proof of lemma \ref{Lm-var-O-Rc-fm} and we
  will denote by "Conj" the conjugate of all therms before this symbol. We
  consider now the local expression
  \begin{eqnarray*}
    \nabla_{g_t}  \dot{J}_t & = & d\,C_{k, \bar{l}} \,\otimes \bar{\zeta}_l^{\ast}
    \otimes \zeta_k\\
    &  & \\
    & + & C_{k, \bar{l}}  \left( \nabla_{g_t}  \bar{\zeta}_l^{\ast} \otimes
    \zeta_k \;\,+ \;\, \bar{\zeta}_l^{\ast} \otimes
    \nabla_{g_t} \zeta_k \right) \;\, + \;\,
    \text{\tmop{Conj}}\\
    &  & \\
    & = &  \left( \partial_p \, C_{k, \bar{l}} \;\, +\;\,
     A^p_{k, r} C_{r, \bar{l}} \right) \zeta_p^{\ast} \otimes
    \bar{\zeta}_l^{\ast} \otimes \zeta_k\\
    &  & \\
    & + & \left( \partial_{\bar{p}} \, C_{k, \bar{l}}
    \;\, -\;\, \overline{A^p_{r, l}} \;C_{k, \bar{r}} \right) 
    \bar{\zeta}_p^{\ast} \otimes \bar{\zeta}_l^{\ast} \otimes \zeta_k
    \;\,+ \;\, \text{\tmop{Conj}} \;,
  \end{eqnarray*}
  obtained by using the expressions (\ref{cov-der-Tg}) and (\ref{cov-der-cTg})
  of the complexified Levi-Civita connection in the appendix. Thus
$$
 \tmop{Tr}_{_{\mathbbm{R}}} \nabla_{g_t}  \dot{J}_t 
     \;\;=\;\;
 \big( \partial_r \,C_{r, \bar{l}} \;\, +\;\,
A^p_{p, r} \,C_{r, \bar{l}} \big)  \;\bar{\zeta}_l^{\ast}
     \;\,+ \;\, \text{\tmop{Conj}} \;, 
$$
and
  \begin{eqnarray*}
d \tmop{Tr}_{_{\mathbbm{R}}} \nabla_{g_t}  \dot{J}_t & = & d\, \big( \partial_r
    \, C_{r, \bar{l}} \;\, + \;\, 
A^p_{p,
    r}\, C_{r, \bar{l}}\big) \wedge \bar{\zeta}_l^{\ast} \;\, +\;\,
    \text{\tmop{Conj}}\\
    &  & \\
    & = & \hspace{0.25em} \partial_{k, r}^2 \,C_{r, \bar{l}}\; \zeta_k^{\ast}
    \wedge \bar{\zeta}_l^{\ast} \;\, + \;\,
    \partial_{\bar{k}, r}^2 \,C_{r, \bar{l}}  \;\bar{\zeta}_k^{\ast} \wedge
    \bar{\zeta}_l^{\ast}\\
    &  & \\
    & + &  d \left( A^p_{p, r} C_{r, \bar{l}} \right) \wedge
    \bar{\zeta}_l^{\ast} \;\,+ \;\,
    \text{\tmop{Conj}}\\
    &  & \\
    & = & \partial_{k, r}^2 \,C_{r, \bar{l}} \; \zeta_k^{\ast}
    \wedge \bar{\zeta}_l^{\ast} \;\,+\;\, d \left[ (\partial_p 
    \, \omega_{r, \bar{s}}) \, \omega^{s, \bar{p}}\,
    C_{r, \bar{l}} \right] \wedge \bar{\zeta}_l^{\ast} \;\, +\;\,
    \text{\tmop{Conj}} \;,
  \end{eqnarray*}
  since the therm
  \begin{eqnarray*}
    \partial_{\bar{k}, r}^2  \, C_{r, \bar{l}} \;\;
    = \;\; \partial_{r, \bar{k}}^2  \, C_{r, \bar{l}} \;\;=\;\; \partial_{r,
    \bar{l}}^2  \; C_{r, \bar{k}}\;,
  \end{eqnarray*}
  is symmetric in the indices $k$ and $l$. The last equality follows from the
  local identity
  \begin{equation}
    \label{loc-dbar-vrJ} \partial_{\bar{k}}  \; C_{r, \bar{l}} \;\;=\;\;
    \partial_{\bar{l}}  \; C_{r, \bar{k}}\;,
  \end{equation}
  which is equivalent to the identity
$$
\bar{\partial}_{_{T_{X, J_t}}} \dot{J}_t \;\;\equiv\;\; 0\;.
$$ 
(See lemma 7 in \cite{Pal1}.) Thus at the point $p_0$ where the geodesic 
coordinates are centered
hold the expression
\begin{eqnarray*}
d \tmop{Tr}_{_{\mathbbm{R}}} \nabla_{g_t}  \dot{J}_t & = & \partial_{k,
    r}^2 \,C_{r, \bar{l}} \; \zeta_k^{\ast} \wedge
    \bar{\zeta}_l^{\ast} 
\\
\\
& +&   \left(
    \partial_{p, \bar{k}}^2 \;\omega_{r, \bar{p}} \;C_{r, \bar{l}} \right)
     \bar{\zeta}_k^{\ast} \wedge \bar{\zeta}_l^{\ast}
    \;\, + \;\, \text{\tmop{Conj}}\;,
\end{eqnarray*}
  i.e.
\begin{eqnarray}  
    \label{dTrDJ} d \tmop{Tr}_{_{\mathbbm{R}}} \nabla_{g_t}  \dot{J}_t 
&=&
    \left( \partial_{k, r}^2 \,C_{r, \bar{l}} \;\, -\;\,
 \partial_{\bar{l}, \bar{r}}^2  \;\overline{C_{r, \bar{k}}}
    \right) \zeta_k^{\ast} \wedge \bar{\zeta}_l^{\ast} \nonumber
\\\nonumber
\\
&-&
 R_{k, \bar{l}} \;C_{l,
    \bar{l}} \;\zeta_k^{\ast} \wedge \zeta_l^{\ast} \;\,-\;\, R_{l, \bar{k}}\; C_{l,
    \bar{l}}  \;\bar{\zeta}_k^{\ast} \wedge \bar{\zeta}_l^{\ast} \;.
\end{eqnarray}
  {\tmstrong{(B)}} In this part we prove the identity (\ref{dec-O-RcI}). We
  keep the previous notations for the geodesic coordinates and we consider
  arbitrary vector fields $\xi, \eta$ with constant coefficients with respect
  to this coordinates. Using the local expression (\ref{dTrDJ}) we infer the
  identity
  \begin{eqnarray*}
    d \tmop{Tr}_{_{\mathbbm{R}}} \nabla_{g_t}  \dot{J}_t \,(J_t \xi, \eta) & = &
    i \left( \partial_{k, r}^2 \;C_{r, \bar{l}} 
    \;\,- \;\, \partial_{\bar{l}, \bar{r}}^2  \;\overline{C_{r, \bar{k}}}
    \right) \left( \xi_k\,  \bar{\eta}_l \;\,+\;\, \bar{\xi}_l\, \eta_k \right)\\
    &  & \\
    & + & i \,R_{k, \bar{l}}\; C_{l, \bar{l}}\;  \big( \xi_l \,\eta_k \;\,-\;\, \xi_k \,\eta_l
    \big) 
\\
\\
&+& i\, R_{l, \bar{k}}\; C_{l, \bar{l}}\;  \left( \bar{\xi}_k\, 
    \bar{\eta}_l \;\,-\;\, \bar{\xi}_l\,  \bar{\eta}_k \right)\;,
  \end{eqnarray*}
  at the point $p_0$. On the
  other hand the general identity
$$
 2\, ( \dot{g}_t^{\ast})_{J_t}^{0, 1} \;\;=\;\; -\;\, J_t  \dot{J}_t \;\,-\;\, (J_t 
     \dot{J}_t)_{g_t}^T\;, 
$$
  (see lemma 3, identity 15 in \cite{Pal1}) combined with the symmetry
  assumption $\dot{J}_t \;=\; ( \dot{J}_t)_{g_t}^T$ implies
  \begin{eqnarray*}
    ( \dot{g}''_t)^{\ast} \;\; = \;\; ( \dot{g}^{\ast}_t)_{_{J_t}}^{0, 1} \;\;=\;\; -\;\, J_t 
    \dot{J}_t \;.
  \end{eqnarray*}
  Using this last identity we expand the therm
  \begin{eqnarray*}
    \tmop{div}_{g_t} \mathcal{D}_{g_t}  \dot{g}''_t (\xi, \eta) \;\; = \;\; 2\,
    \nabla_{g_t} \mathcal{D}_{g_t}  \dot{g}''_t  \left( \zeta_l,
    \bar{\zeta}_l, \xi, \eta \right) \;\,+\;\, 2\, \nabla_{g_t} \mathcal{D}_{g_t} 
    \dot{g}''_t  \left( \bar{\zeta}_l, \zeta_l, \xi, \eta \right) \;.
  \end{eqnarray*}
  We obtain the equalities
  \begin{eqnarray*}
    2\, \nabla_{g_t} \mathcal{D}_{g_t}  \dot{g}''_t  \left( \zeta_l,
    \bar{\zeta}_l, \xi, \eta \right) & = & -\;\, 2\, \nabla^2_{g_t}  \dot{g}''_t 
    \left( \zeta_l, \bar{\zeta}_l, \xi, \eta \right) \;\,+\;\, 2\, \nabla^2_{g_t} 
    \dot{g}''_t  \left( \zeta_l, \xi, \bar{\zeta}_l, \eta \right) \\
    &  & \\
    & + & 2 \,\nabla^2_{g_t}  \dot{g}''_t  \left( \zeta_l, \eta, \bar{\zeta}_l,
    \xi \right) \\
    &  & \\
    & = & - \;\,2\, g \left( J_t \nabla^2_{g_t}  \dot{J}_t \left( \zeta_l, \xi,
    \bar{\zeta}_l \right), \eta \right) 
\\
\\
&-& 2\, g \left( J_t \nabla^2_{g_t} 
    \dot{J}_t \left( \zeta_l, \eta, \bar{\zeta}_l \right), \xi \right)\\
    &  & \\
    & + & 2 \,g \left( J_t \nabla^2_{g_t}  \dot{J}_t \left( \zeta_l,
    \bar{\zeta}_l, \xi \right), \eta \right) \;.
  \end{eqnarray*}
  We compute now the local expression of the tensor $\nabla_{g_t}^2 
  \dot{J}_t$ . Taking the covariant derivative of the local expression of the
  tensor $\nabla_{g_t}  \dot{J}_t$, obtained in the beginning of step
  {\tmstrong{(A)}}, we infer the expression at the point $p_0$
  \begin{eqnarray*}
    \nabla_{g_t}^2  \dot{J}_t & = & \left( d\, \partial_p \,C_{k, \bar{l}} \;\,+\;\, d\,
    A_{k, h}^p \,C_{h, \bar{l}} \right) \otimes \zeta_p^{\ast} \otimes
    \bar{\zeta}_l^{\ast} \otimes \zeta_k \\
    &  & \\
    & + & \left( d\, \partial_{\bar{p}}\, C_{k, \bar{l}} \;\,-\;\, d\, \overline{A_{h,
    l}^p} \,C_{k, \bar{h}} \right) \otimes \bar{\zeta}_p^{\ast} \otimes
    \bar{\zeta}_l^{\ast} \otimes \zeta_k \;\,+\;\, \tmop{Conj}\\
    &  & \\
    & = & \partial_{r, p}^2 \,C_{k, \bar{l}} \;\zeta_r^{\ast} \otimes
    \zeta_p^{\ast} \otimes \bar{\zeta}_l^{\ast} \otimes \zeta_k \\
    &  & \\
    & + & \left( \partial_{\bar{r}, p}^2\, C_{k, \bar{l}} \;\,+\;\, C_{l, \bar{l}}\;
    \varphi_{p, \bar{r}, l, \bar{k}} \right)  \bar{\zeta}_r^{\ast} \otimes
    \zeta_p^{\ast} \otimes \bar{\zeta}_l^{\ast} \otimes \zeta_k\\
    &  & \\
    & + & \left( \partial_{r, \bar{p}}^2 \,C_{k, \bar{l}} \;\,-\;\, C_{k, \bar{k}}\;
    \varphi_{r, \bar{p}, k, \bar{l}} \right) \zeta_r^{\ast} \otimes
    \bar{\zeta}_p^{\ast} \otimes \bar{\zeta}_l^{\ast} \otimes \zeta_k \\
    &  & \\
    & + & \partial_{\bar{r}, \bar{p}}^2\, C_{k, \bar{l}} \; \bar{\zeta}_r^{\ast}
    \otimes \bar{\zeta}_p^{\ast} \otimes \bar{\zeta}_l^{\ast} \otimes \zeta_k
    \;\,+\;\, \tmop{Conj} \;.
  \end{eqnarray*}
  Using this last expression we expand the therms (i.e the sums, by the
  Einstein convention)
  \begin{eqnarray*}
    J_t \nabla^2_{g_t}  \dot{J}_t \left( \zeta_l, \xi, \bar{\zeta}_l \right) &
    = &  i  \left( \partial_{l, p}^2\, C_{k, \bar{l}} \;\,+\;\, C_{l, \bar{l}}\;
    \varphi_{p, \bar{l}, l, \bar{k}} \right) \xi_p\,\zeta_k
\\
\\
& +& i\left( \partial_{l,
    \bar{p}}^2 \,C_{k, \bar{l}} \;\,+\;\, C_{k, \bar{k}} \;R_{k, \bar{p}} \right) 
    \bar{\xi}_p  \,\zeta_k,\\
    &  & \\
    J_t \nabla^2_{g_t}  \dot{J}_t \left( \zeta_l, \bar{\zeta}_l, \xi \right) &
    = & i \left( \partial_{l, \bar{l}}^2\; C_{k, \bar{p}} \;\,+\;\, C_{k, \bar{k}}\; R_{k,
    \bar{p}} \right)  \bar{\xi}_p \,\zeta_k 
\\
\\
&-& i \left( \partial_{l, \bar{l}}^2 \;
    \overline{C_{k, \bar{p}}} \;\,-\;\, C_{p, \bar{p}} \;R_{k, \bar{p}} \right) \xi_p \,
    \bar{\zeta}_k \;.
  \end{eqnarray*}
  We infer the expression
  \begin{eqnarray*}
    &&2 \nabla_{g_t} \mathcal{D}_{g_t}  \dot{g}''_t  \left( \zeta_l,
    \bar{\zeta}_l, \xi, \eta \right) 
\\
\\
& = & - \;\,i \Big[\left( \partial_{l, p}^2\,
    C_{k, \bar{l}} \;\,+\;\, C_{l, \bar{l}}\; \varphi_{p, \bar{l}, l, \bar{k}} \right)
    \xi_p \;\,+\;\, \left( \partial_{l, \bar{p}}^2\, C_{k, \bar{l}}\;\, +\;\, C_{k, \bar{k}}\;
    R_{k, \bar{p}} \right)  \bar{\xi}_p\Big]  \bar{\eta}_k\\
    &  & \\
    & - & i \Big[\left( \partial_{l, p}^2 \,C_{k, \bar{l}} \;\,+\;\, C_{l, \bar{l}}\;
    \varphi_{p, \bar{l}, l, \bar{k}} \right) \eta_p \;\,+\;\, \left( \partial_{l,
    \bar{p}}^2 \,C_{k, \bar{l}} \;\,+\;\, C_{k, \bar{k}}\; R_{k, \bar{p}} \right) 
    \bar{\eta}_p\Big]  \bar{\xi}_k\\
    &  & \\
    & + & i \left( \partial_{l, \bar{l}}^2\; C_{k, \bar{p}} \;\,+\;\, C_{k, \bar{k}}\;
    R_{k, \bar{p}} \right)  \bar{\xi}_p \, \bar{\eta}_k \;\,-\;\, i \left( \partial_{l,
    \bar{l}}^2  \;\overline{C_{k, \bar{p}}} \;\,-\;\, C_{p, \bar{p}} \;R_{k, \bar{p}}
    \right) \xi_p \,\eta_k\\
    &  & \\
    & = & - \;\,i \left( \partial_{l, p}^2\, C_{k, \bar{l}} \;\,+\;\, C_{l, \bar{l}}\;
    \varphi_{p, \bar{l}, l, \bar{k}} \right) \xi_p \, \bar{\eta}_k\\
    &  & \\
    & - & i \Big[ \left( \partial_{l, p}^2 \,C_{k, \bar{l}} \;\,+\;\, C_{l, \bar{l}}\;
    \varphi_{p, \bar{l}, l, \bar{k}} \right) \eta_p \;\,+\;\, \left( \partial_{l,
    \bar{l}}^2 \;C_{k, \bar{p}} \;\,+\;\, C_{k, \bar{k}} \;R_{k, \bar{p}} \right) 
    \bar{\eta}_p\Big]  \bar{\xi}_k\\
    &  & \\
    & - & i \left( \partial_{l, \bar{l}}^2  \;\overline{C_{k, \bar{p}}} \;\,-\;\, C_{p,
    \bar{p}}\; R_{k, \bar{p}} \right) \xi_p \,\eta_k\;,
  \end{eqnarray*}
  thanks to the identity (\ref{loc-dbar-vrJ}). Moreover simplifying the therms
  \begin{eqnarray*}
    C_{l, \bar{l}} \;\varphi_{p, \bar{l}, l, \bar{k}}\; \xi_p  \,\bar{\eta}_k \;,\qquad
    C_{l, \bar{l}} \;\varphi_{p, \bar{l}, l, \bar{k}} \;\eta_p  \,\bar{\xi}_k\;,
  \end{eqnarray*}
  with their conjugates (We remind that $C_{l, \bar{l}} \;=\; \overline{C_{l,
  \bar{l}}}$ thanks to our special choice of geodesic coordinates.) and rearranging the conjugate therms we obtain
  \begin{eqnarray*}
    \tmop{div}_{g_t} \mathcal{D}_{g_t}  \dot{g}''_t (\xi, \eta) & = & - \;\,i\,
    \partial_{l, p}^2 \,C_{k, \bar{l}}  \left( \xi_p \, \bar{\eta}_k \;\,+\;\, \bar{\xi}_k\,
    \eta_p \right) 
\\
\\
&-& i \left( \partial_{l, \bar{l}}^2\; C_{k, \bar{p}} \;\,+\;\, C_{k,
    \bar{k}}\; R_{k, \bar{p}} \right)  \bar{\xi}_k \, \bar{\eta}_p\\
    &  & \\
    & + & i \left( \partial_{l, \bar{l}}^2\; C_{k, \bar{p}} \;\,-\;\, C_{p, \bar{p}}\;
    R_{p, \bar{k}} \right)  \bar{\xi}_p  \,\bar{\eta}_k \;\,+\;\, \tmop{Conj}\\
    &  & \\
    & = & - \;\,i \,\partial_{l, p}^2 \,C_{k, \bar{l}}  \left( \xi_p \, \bar{\eta}_k \;\,+\;\,
    \bar{\xi}_k \,\eta_p \right) 
\\
\\
&+& i \left( \partial_{l, \bar{l}}^2 \;C_{p,
    \bar{k}} \;\,-\;\, \partial_{l, \bar{l}}^2 \;C_{k, \bar{p}} \right)  \bar{\xi}_k \,
    \bar{\eta}_p \\
    &  & \\
    & - & 2\, i\, C_{k, \bar{k}} \;R_{k, \bar{p}}\;  \bar{\xi}_k \, \bar{\eta}_p \;\,+\;\,
    \tmop{Conj}\\
    &  & \\
    & = & - \;\,i\, \partial_{l, p}^2 \,C_{k, \bar{l}}  \left( \xi_p  \,\bar{\eta}_k \;\,+\;\,
    \bar{\xi}_k \,\eta_p \right) 
\\
\\
&+& i \left( \partial_{l, \bar{l}}^2\; C_{p,
    \bar{k}} \;\,-\;\, \partial_{l, \bar{l}}^2\; C_{k, \bar{p}} \right)  \bar{\xi}_k\, 
    \bar{\eta}_p \;\,+\;\, \tmop{Conj}\\
    &  & \\
    & - & 2 \tmop{Ric}_{_{J_t}} (\omega_t) ( \dot{J}_t \xi, \eta)\;,
  \end{eqnarray*}
  at the point $p_0$. We remind now that for any tensor $\alpha$ and any
  smooth vector fields $v, w$ such that $\nabla_{g_t} w (q) = 0$ at some point
  $q$ hold the well known identity $\nabla^2_{v, w} \,\alpha\,(q) \;=\; \nabla_v \nabla_w\,
  \alpha\,(q)$. By applying this to the identity
  $\dot{J}_t = ( \dot{J}_t)_{g_t}^T$ we infer the formula
  \begin{eqnarray*}
    \nabla_{V, W}^2  \,\dot{J}_t & = & \left( \nabla_{V, W}^2  \,\dot{J}_t
    \right)_{g_t}^T\;,
  \end{eqnarray*}
  for any smooth vector fields $V, W.$ In particular for all indices $r, p,$
  the identity
  \begin{eqnarray*}
    \nabla_{\zeta_r, \zeta_p}^2  \dot{J}_t & = & \left( \nabla_{\zeta_r,
    \zeta_p}^2  \dot{J}_t \right)_{g_t}^T,
  \end{eqnarray*}
  is equivalent to the local symmetry identity
  \begin{eqnarray*}
    \partial_{r, p}^2 \,C_{k, \bar{l}}  & = & \partial_{r, p}^2\, C_{l, \bar{k}}\;,
  \end{eqnarray*}
  for all indices $r, p, k, l$. Moreover for all indices $l,$ the identity
  \begin{eqnarray*}
    \nabla_{\zeta_l, \bar{\zeta}_l}^2  \,\dot{J}_t & = & \left( \nabla_{\zeta_l,
    \bar{\zeta}_l}^2  \,\dot{J}_t \right)_{g_t}^T\;,
  \end{eqnarray*}
  implies the symmetry identity
  \begin{equation}
    \label{sc-ord-sm-cx} \partial_{l, \bar{l}}^2 \;C_{k, \bar{p}} \;\,+\;\, C_{k,
    \bar{k}} \;R_{k, \bar{p}} \;\;=\;\; \partial_{l, \bar{l}}^2 \;C_{p, \bar{k}} \;\,+\;\, C_{p,
    \bar{p}}\; R_{p, \bar{k}}\;,
  \end{equation}
  for all indices $k, p$. We infer the expressions
  \begin{eqnarray*}
    \tmop{div}_{g_t} \mathcal{D}_{g_t}  \dot{g}''_t (\xi, \eta) & = & - \;\,i\,
    \partial_{l, p}^2 \,C_{l, \bar{k}}  \left( \xi_p  \,\bar{\eta}_k \;\,+\;\, \bar{\xi}_k\,
    \eta_p \right) 
\\
\\
&+& i \left( C_{k, \bar{k}} \;R_{k, \bar{p}} \;\,-\;\, C_{p, \bar{p}}\;
    R_{p, \bar{k}} \right)  \bar{\xi}_k  \,\bar{\eta}_p \;\,+\;\, \tmop{Conj}\\
    &  & \\
    & - & 2 \,\tmop{Ric}_{_{J_t}} (\omega_t) ( \dot{J}_t \xi, \eta)\\
    &  & \\
    & = & - \;\,i\, \partial_{p, l}^2\, C_{l, \bar{k}}  \left( \xi_p \, \bar{\eta}_k \;\,+\;\,
    \bar{\xi}_k \,\eta_p \right) 
\\
\\
&+& i\, C_{k, \bar{k}} \;R_{k, \bar{p}}  \left(
    \bar{\xi}_k  \,\bar{\eta}_p \;\,-\;\, \bar{\xi}_p \, \bar{\eta}_k \right) \;\,+\;\,
    \tmop{Conj}\\
    &  & \\
    & - & 2 \,\tmop{Ric}_{_{J_t}} (\omega_t) ( \dot{J}_t \xi, \eta)\\
    &  & \\
    & = & - \;\,d \tmop{Tr}_{_{\mathbbm{R}}} \nabla_{g_t}  \dot{J}_t \,(J_t \xi,
    \eta) \;\,-\;\, 2\, \tmop{Ric}_{_{J_t}} (\omega_t) ( \dot{J}_t \xi, \eta)\;,
  \end{eqnarray*}
  thanks to the expression of the therm $d \tmop{Tr}_{_{\mathbbm{R}}}
  \nabla_{g_t}  \dot{J}_t \,(J_t \xi, \eta)$ obtained in the beginning of step
  {\tmstrong{(B)}}. We deduce the required identity (\ref{dec-O-RcI}).

  {\tmstrong{(C)}} We show now the formula (\ref{dec-O-Rc2}). We decompose the
  therm
  \begin{eqnarray*}
    -\;\,\mathcal{D}_{g_t}  \dot{g}''_t (\nabla_{g_t} f_t, \xi, \eta) & = &
    \nabla_{g_t}  \dot{g}''_t (\nabla_{g_t} f_t, \xi, \eta) \;\,-\;\, \nabla_{g_t} 
    \dot{g}''_t (\xi, \nabla_{g_t} f_t, \eta) 
\\
\\
&-& \nabla_{g_t}  \dot{g}''_t
    (\eta, \nabla_{g_t} f_t, \xi)\\
    &  & \\
    & = & - \;\,g_t \left( J_t \nabla_{g_t}  \dot{J}_t (\nabla_{g_t} f_t, \xi),
    \eta \right) 
\\
\\
&+& g_t \left( J_t \nabla_{g_t}  \dot{J}_t (\xi, \nabla_{g_t}
    f_t), \eta \right) \\
    &  & \\
    & + & g_t \left( J_t \nabla_{g_t}  \dot{J}_t (\eta, \nabla_{g_t} f_t),
    \xi \right) \\
    &  & \\
    & = & -\;\, g_t \left( J_t \nabla^{1, 0}_{g_t, J_t}  \dot{J}_t (\nabla_{g_t}
    f_t, \xi), \eta \right) 
\\
\\
&+& g_t \left( J_t \nabla^{1, 0}_{g_t, J_t} 
    \dot{J}_t (\xi, \nabla_{g_t} f_t), \eta \right)\\
    &  & \\
    & + & g_t \left( J_t \nabla^{1, 0}_{g_t, J_t}  \dot{J}_t (\eta,
    \nabla_{g_t} f_t), \xi \right) 
\\
\\
&+& g_t \left( J_t \nabla^{0, 1}_{g_t, J_t} 
    \dot{J}_t (\nabla_{g_t} f_t, \eta), \xi \right),
  \end{eqnarray*}
  thanks to the symmetry property (\ref{Tg-cx-str}). We observe also that the
  endomorphism $A \assign J_t \dot{J}_t$ satisfies
  the symmetry identities
\begin{eqnarray}
\label{cx-sm-A1} \xi \;\neg\; \nabla_{g, J}^{1, 0}\, A &=& \left( \xi \;\neg\;
    \nabla_{g, J}^{1, 0} \,A \right)_g^T\;,
\\\nonumber
\\
\label{cx-sm-A2} \xi \;\neg\; \nabla_{g, J}^{0, 1} \,A &=& \left( \xi \;\neg\;
    \nabla_{g, J}^{0, 1} \,A \right)_g^T\;,
  \end{eqnarray}
  which are direct consequence of the equality $A = A_g^T$. We infer
  \begin{eqnarray*}
    g_t \left( J_t \nabla^{1, 0}_{g_t, J_t}  \dot{J}_t (\xi, \nabla_{g_t}
    f_t), \eta \right) & = & g_t \left( J_t \nabla^{1, 0}_{g_t, J_t} 
    \dot{J}_t (\xi, \eta), \nabla_{g_t} f_t \right) \;,\\
    &  & \\
    g_t \left( J_t \nabla^{0, 1}_{g_t, J_t}  \dot{J}_t (\nabla_{g_t} f_t,
    \eta), \xi \right) & = & g_t \left( J_t \nabla^{0, 1}_{g_t, J_t} 
    \dot{J}_t (\nabla_{g_t} f_t, \xi), \eta \right) \;.
  \end{eqnarray*}
  Thus
  \begin{eqnarray*}
    -\;\,\mathcal{D}_{g_t}  \dot{g}''_t (\nabla_{g_t} f_t, \xi, \eta) & = & d f_t
    \cdot J_t \left( \nabla^{1, 0}_{g_t, J_t}  \dot{J}_t (\xi, \eta) \;\,+\;\,
    \nabla^{1, 0}_{g_t, J_t}  \dot{J}_t (\eta, \xi) \right)\\
    &  & \\
    & - & g_t \left( \nabla_{g_t}  \dot{J}_t (J_t \nabla_{g_t} f_t, \xi),
    \eta \right) \;.
  \end{eqnarray*}
  We use this last identity in the decomposition of the therm
  \begin{eqnarray*}
    \left( d \,\dot{J}^{\ast}_t d \,f_t \right) (J_t \xi, \eta) & = & \nabla_{g_t}
    d \,f_t \,(J_t \xi, \dot{J}_t \eta) \;\,+\;\, d\, f_t \cdot \nabla_{g_t}  \dot{J}_t \,(J_t
    \xi, \eta)\\
    &  & \\
    & - & \nabla_{g_t} d \,f_t \,(\eta, \dot{J}_t J_t \xi) \;\,-\;\, d\, f_t \cdot
    \nabla_{g_t}  \dot{J}_t \,(\eta, J_t \xi)\\
    &  & \\
    & = & g_t \left( \nabla^2_{g_t} f_t \, J_t \xi, \dot{J}_t \eta \right)
    \;\,+\;\, \nabla_{g_t} d\, f_t (J_t  \dot{J}_t \xi, \eta)\\
    &  & \\
    & + & d \,f_t \cdot J_t \left( \nabla^{1, 0}_{g_t, J_t}  \dot{J}_t \,(\xi,
    \eta) \;\,+\;\, \nabla^{1, 0}_{g_t, J_t}  \dot{J}_t \,(\eta, \xi) \right)\\
    &  & \\
    & = & g_t \left( \dot{J}_t \nabla^2_{g_t} f_t \, J_t \xi, \eta \right)
    \;\,+\;\, \nabla_{g_t} d \,f_t \,(J_t  \dot{J}_t \xi, \eta)\\
    &  & \\
    & - & \mathcal{D}_{g_t}  \dot{g}''_t \,(\nabla_{g_t} f_t, \xi, \eta) \;\,+\;\, g_t
    \left( \nabla_{g_t}  \dot{J}_t \,(J_t \nabla_{g_t} f_t, \xi), \eta \right) \;.
  \end{eqnarray*}
  Combining this with the decomposition formula (\ref{cx-dec-Hess}) we infer
  the required identity (\ref{dec-O-Rc2}). 
\end{proof}

Combining the time derivative of the identity (\ref{cx-dec-Ric}) with the
variation formula (\ref{var-Om-Ric}) and with corollary \ref{rm-vr-Rc-fm} we
infer immediately the following corollary.

\begin{corollary}
  \label{vr-ant-Hess}Under the the assumptions of lemma \ref{Lm-var-O-Rc-fm}
  hold the identity
\begin{eqnarray*}
    2 \,\frac{d}{dt}  \left( g_t \,\bar{\partial}_{_{T_{X, J}}} \nabla_{g_t}
    f_t \right) & = & \tmop{div}^{^{_{_{\Omega}}}}_{g_t} \mathcal{D}_{g_t} 
    \dot{g}'_t\\
    &  & \\
    & - & g_t \left( J_t \nabla_{g_t} f_t \;\neg\; \nabla_{g_t}  \dot{J}_t \;\,+\;\,
    \dot{J}_t \nabla^2_{g_t} f_t \,J_t \;\,-\;\, J_t \nabla^2_{g_t} f_t \, \dot{J}_t
    \right) \;,
  \end{eqnarray*}
where $ \dot{g}'_t \assign ( \dot{g}_t)_{_{J_t}}'$ and $f_t \assign \log \frac{d V_{g_t}}{\Omega}$.
\end{corollary}

\section{Weitzenb\"ock type formulas}

We denote by $P_g^{\ast}$ the formal adjoint of an operator $P_g$ depending on
$g$. We observe that the operator
\begin{eqnarray*}
  P^{\ast_{_{\Omega}}}_g \;\; : = \;\; e^f P^{\ast}_g  \left( e^{- f} \bullet
  \right)\;,
\end{eqnarray*}
is the formal adjoint of $P_g$ with respect to the scalar product $\int_X
\left\langle \cdot, \cdot \right\rangle_g \Omega$. With this notation hold the identity
$\nabla_g^{\ast_{_{\Omega}}} \;=\; -\, \tmop{div}^{^{_{_{\Omega}}}}_g$. As in \cite{Pal2} we define the $\Omega$-Laplacian
\begin{eqnarray*}
  \Delta^{^{_{_{\Omega}}}}_g \;\; \assign \;\; \nabla_g^{\ast_{_{\Omega}}} \nabla_g
  \;\;=\;\; \Delta_g \;\,+\;\, \nabla_g f \;\neg\; \nabla_g\;,
\end{eqnarray*}
with $f \assign \log \frac{d V_g}{\Omega}$ and we remind the following
result obtained in \cite{Pal2}.

\begin{lemma}
  \label{Endo-Div}For any $g \in \mathcal{M}$ and $u \in C^{\infty} (X, S^2
  T^{\ast}_X)$ hold the formula
  \begin{eqnarray*}
    \left( \tmop{div}^{^{_{_{\Omega}}}}_g \mathcal{D}_g u \right)_g^{\ast}  &
    = & \frac{1}{2} \,\nabla^{\ast_{_{\Omega}}}_{_{T_X, g}} \nabla_{_{T_X, g}}
    u^{\ast}_g \;\,+\;\, \frac{1}{2}\,  \left( \nabla^{\ast_{_{\Omega}}}_{_{T_X, g}}
    \nabla_{_{T_X, g}} u^{\ast}_g \right)_g^T \;\,-\;\, \Delta^{^{_{_{\Omega}}}}_g
    u^{\ast}_g\;,
  \end{eqnarray*}
  where $\nabla_{_{T_X, g}}$ denotes the covariant exterior derivative acting
  on $T_X$-valued differential forms.
\end{lemma}

Using the expressions (\ref{expr-adDel}) and (\ref{expr-adDbar}) of the adjoint of the standard complex operators in the
subsection \ref{cpx-op} of the appendix we can show the following complex
Weitzenb\"ock type formulas.

\begin{lemma}
  \label{Cx-Weitz}Consider $(J, g) \in \mathcal{K}\mathcal{S}$ and let $A \in
  C^{\infty} (X, T^{\ast}_{X, - J} \otimes_{_{\mathbbm{C}}} T_{X, J}), B \in
  C^{\infty} (X, T^{\ast}_{X, J} \otimes_{_{\mathbbm{C}}} T_{X, J})$ be
  endomorphism sections. Then hold the identities
  \begin{eqnarray}
\label{weitz-A} \Delta_g A &=& \partial^{\ast_g}_{_{T_{X, J}}}
    \partial^g_{_{T_{X, J}}} A \;\,-\;\, A \tmop{Ric}_g^{\ast} \;\,-\;\, \tmop{Ric}_g^{\ast}
    A\;,
\\\nonumber
\\
\label{weitz-B} \Delta_g B &= &\bar{\partial}^{\ast_g}_{_{T_{X, J}}} 
    \bar{\partial}_{_{T_{X, J}}} B \;\,-\;\, B \tmop{Ric}_g^{\ast} \;\,+\;\,
    \tmop{Ric}_g^{\ast} B \;.
\end{eqnarray}
\end{lemma}
\begin{proof}
  We choose $g$-orthonormal and $J$-holomorphic coordinates centered at an
  arbitrary point $p_0$. Let $A \;=\; C_{k, \bar{l}} \; \bar{\zeta}_l^{\ast} \otimes
  \zeta_k \;+\; \tmop{Conj}$, be the local expression of $A$ and we consider the
  local expression
  \begin{eqnarray*}
    \nabla_g A & = & d\,C_{k, \bar{l}} \,\otimes \bar{\zeta}_l^{\ast} \otimes
    \zeta_k\\
    &  & \\
    & + & C_{k, \bar{l}}  \left( \nabla_{g_t}  \bar{\zeta}_l^{\ast} \otimes
    \zeta_k \;\,+ \;\, \bar{\zeta}_l^{\ast} \otimes
    \nabla_{g_t} \zeta_k \right) \;\, + \;\,
    \text{\tmop{Conj}}\\
    &  & \\
    & = &  \left( \partial_p \, C_{k, \bar{l}} \;\,+\;\,
    A^p_{k, r} \;C_{r, \bar{l}} \right) \zeta_p^{\ast} \otimes
    \bar{\zeta}_l^{\ast} \otimes \zeta_k\\
    &  & \\
    & + & \left ( \partial_{\bar{p}} \; C_{k, \bar{l}}
    \;\, -\;\, \overline{A^p_{r, l}}\; C_{k, \bar{r}} \right) 
    \bar{\zeta}_p^{\ast} \otimes \bar{\zeta}_l^{\ast} \otimes \zeta_k
    \;\,+ \;\, \text{\tmop{Conj}} \;,
  \end{eqnarray*}
  obtained by using by the expression (\ref{cov-der-Tg}) of the complexified
  Levi-Civita connection in the appendix. Taking again the covariant
  derivative we infer the expression at the point $p_0$
  \begin{eqnarray*}
    \nabla_g^2 A & = & \left( d \,\partial_p \,C_{k, \bar{l}} \;\,+\;\, d\, A_{k, h}^p \;C_{h,
    \bar{l}} \right) \otimes \zeta_p^{\ast} \otimes \bar{\zeta}_l^{\ast}
    \otimes \zeta_k \\
    &  & \\
    & + & \left( d \,\partial_{\bar{p}} \;C_{k, \bar{l}} \;\,-\;\, d\, \overline{A_{h,
    l}^p} \;C_{k, \bar{h}} \right) \otimes \bar{\zeta}_p^{\ast} \otimes
    \bar{\zeta}_l^{\ast} \otimes \zeta_k \;\,+\;\, \tmop{Conj}\\
    &  & \\
    & = & \partial_{r, p}^2 \,C_{k, \bar{l}} \;\zeta_r^{\ast} \otimes
    \zeta_p^{\ast} \otimes \bar{\zeta}_l^{\ast} \otimes \zeta_k \\
    &  & \\
    & + & \left( \partial_{\bar{r}, p}^2\, C_{k, \bar{l}} \;\,+\;\, C_{h, \bar{l}}\;
    \varphi_{p, \bar{r}, h, \bar{k}} \right)  \bar{\zeta}_r^{\ast} \otimes
    \zeta_p^{\ast} \otimes \bar{\zeta}_l^{\ast} \otimes \zeta_k\\
    &  & \\
    & + & \left( \partial_{r, \bar{p}}^2 \,C_{k, \bar{l}} \;\,-\;\, C_{k, \bar{h}}\;
    \varphi_{r, \bar{p}, h, \bar{l}} \right) \zeta_r^{\ast} \otimes
    \bar{\zeta}_p^{\ast} \otimes \bar{\zeta}_l^{\ast} \otimes \zeta_k \\
    &  & \\
    & + & \partial_{\bar{r}, \bar{p}}^2\; C_{k, \bar{l}}  \;\bar{\zeta}_r^{\ast}
    \otimes \bar{\zeta}_p^{\ast} \otimes \bar{\zeta}_l^{\ast} \otimes \zeta_k
    \;\,+\;\, \tmop{Conj}\; .
  \end{eqnarray*}
  Thus
  \begin{eqnarray*}
    \Delta_g A & = & - \;\,2\, \nabla_{g, \zeta_k} \nabla_{g, \bar{\zeta}_k} A \;\,-\;\, 2\,
    \nabla_{g, \bar{\zeta}_k} \nabla_{g, \zeta_k} A\\
    &  & \\
    & = & - \;\,2\, \left( 2 \,\partial^2_{r, \bar{r}} \;C_{k, \bar{l}} \;\,+\;\, C_{k,
    \bar{h}} \;R_{h, \bar{l}} \;\,-\;\, C_{h, \bar{l}} \;R_{h, \bar{k}}  \right) 
    \bar{\zeta}_l^{\ast} \otimes \zeta_k \;\,+\;\, \tmop{Conj},
  \end{eqnarray*}
  at the point $p_0$. Moreover using the local expression
  \begin{eqnarray*}
    \partial^g_{_{T_{X, J}}} \zeta_k \;\; = \;\, A_{l, k}^p\; \zeta^{\ast}_p \otimes
    \zeta_l \;\,+\;\, \tmop{Conj}\;,
  \end{eqnarray*}
  we obtain
\begin{eqnarray*}
    \partial^g_{_{T_{X, J}}} A \;\; = \;\; \left( \partial_p \,C_{k, \bar{l}} \;\,+\;\, A_{k,
    h}^p \;C_{h, \bar{l}} \right)  \left( \zeta_p^{\ast} \wedge
    \bar{\zeta}_l^{\ast} \right) \otimes \zeta_k \;\,+\;\, \tmop{Conj}\;,
\end{eqnarray*}
  and
  \begin{eqnarray*}
    \nabla^{0, 1}_{g, J} \,\partial^g_{_{T_{X, J}}} A & = & \left(
    \partial^2_{p, \bar{r}}\; C_{k, \bar{l}} \;\,+\;\, \varphi_{p, \bar{r}, h, \bar{k}}\;
    C_{h, \bar{l}} \right)  \bar{\zeta}_r^{\ast} \otimes \left( \zeta_p^{\ast}
    \wedge \bar{\zeta}_l^{\ast} \right) \otimes \zeta_k \;\,+\;\, \tmop{Conj}\;,
  \end{eqnarray*}
  Thus using the expression (\ref{expr-adDel}) we expand the therm
  \begin{eqnarray*}
    \partial^{\ast_g}_{_{T_{X, J}}} \partial^g_{_{T_{X, J}}} A & = & - \;\,2
    \tmop{Tr}_g  \left( \nabla^{0, 1}_{g, J} \,\partial^g_{_{T_{X, J}}} A
    \right) \\
    &  & \\
    & = & - \;\,4\, \nabla^{0, 1}_{g, J} \,\partial^{g}_{_{T_{X, J}}} A \left(
    \bar{\zeta}_r, \zeta_r, \cdot \right)\\
    &  & \\
    & = & - \;\,4 \left( \partial^2_{r, \bar{r}} \;C_{k, \bar{l}} \;\,-\;\, C_{h, \bar{l}}\;
    R_{h, \bar{k}}  \right)  \bar{\zeta}_l^{\ast} \otimes \zeta_k \;\,+\;\,
    \tmop{Conj} \;.
  \end{eqnarray*}
  We infer the required formula for $A$. Let now $B \;=\; B_{k, \bar{l}}\;
  \zeta^{\ast}_k \otimes \zeta_l \;\,+\;\, \tmop{Conj}$, be the local expression of
  $B$. Expanding the identity
  \begin{eqnarray*}
    \nabla_g B & = & d \,B_{k, \bar{l}} \,\otimes \zeta_k^{\ast} \otimes \zeta_l \;\,+\;\,
    B_{k, \bar{l}} \;\nabla_g \,\zeta_k^{\ast} \otimes \zeta_l \;\,+\;\, B_{k, \bar{l}}\;
    \zeta_k^{\ast} \otimes \nabla_g \,\zeta_l \;\,+\;\, \tmop{Conj}\;,
  \end{eqnarray*}
  we infer
  \begin{eqnarray*}
    \nabla_g B & = & \left( \partial_p\, B_{k, \bar{l}} \;\,+\;\, B_{k, \bar{h}} \;A_{l,
    h}^p \;\,-\;\, B_{h, \bar{l}} \;A_{h, k}^p \right) \zeta_p^{\ast} \otimes
    \zeta_k^{\ast} \otimes \zeta_l 
\\
\\
&+&\partial_{\bar{p}} \;B_{k, \bar{l}} \;
    \bar{\zeta}_p^{\ast} \otimes \zeta_k^{\ast} \otimes \zeta_l \;\,+\;\, \tmop{Conj}\;,
  \end{eqnarray*}
  at the point $p_0$. We deduce the local expression
  \begin{eqnarray*}
    \nabla^2_g B & = & \partial_{r, p}^2 \,B_{k, \bar{l}} \;\zeta_r^{\ast} \otimes
    \zeta_p^{\ast} \otimes \zeta_k^{\ast} \otimes \zeta_l \\
    &  & \\
    & + & \Big( \partial^2_{p, \bar{r}} \;B_{k, \bar{l}} \;\,+\;\, B_{k, \bar{h}}\;
    \varphi_{p, \bar{r}, h, \bar{l}} \;\,-\;\, B_{h, \bar{l}} \;\varphi_{p, \bar{r}, k,
    \bar{h}} \Big)  \bar{\zeta}_r^{\ast} \otimes \zeta_p^{\ast} \otimes
    \zeta_k^{\ast} \otimes \zeta_l\\
    &  & \\
    & + & \partial_{r, \bar{p}}^2 \;B_{k, \bar{l}}\; \zeta_r^{\ast} \otimes
    \bar{\zeta}_p^{\ast} \otimes \zeta_k^{\ast} \otimes \zeta_l \;\,+\;\,
    \partial^2_{\bar{r}, \bar{p}} \;B_{k, \bar{l}} \; \bar{\zeta}_r^{\ast} \otimes
    \bar{\zeta}_p^{\ast} \otimes \zeta_k^{\ast} \otimes \zeta_l \;\,+\;\, \tmop{Conj}\;.
  \end{eqnarray*}
  Thus
  \begin{eqnarray*}
    \Delta_g B & = & - \;\,2\, \nabla_{g, \zeta_k} \nabla_{g, \bar{\zeta}_k} B \;\,-\;\, 2\,
    \nabla_{g, \bar{\zeta}_k} \nabla_{g, \zeta_k} B\\
    &  & \\
    & = & - \;\,2 \left( 2\, \partial^2_{r, \bar{r}} \;B_{k, \bar{l}} \;\,-\;\, B_{k,
    \bar{h}} \;R_{h, \bar{l}} \;\,+\;\, B_{h, \bar{l}} \;R_{k, \bar{h}} \right)
    \zeta_k^{\ast} \otimes \zeta_l \;\,+\;\, \tmop{Conj}\;,
  \end{eqnarray*}
  at the point $p_0$. On the other hand deriving the
  local expression
  \begin{eqnarray*}
    \bar{\partial}_{_{T_{X, J}}} B \;\; = \;\; -\;\, \partial_{\bar{p}}\; B_{k,
    \bar{l}}  \left( \zeta_k^{\ast} \wedge \bar{\zeta}_p^{\ast} \right)
    \otimes \zeta_l \;\,+\;\, \tmop{Conj}\;,
  \end{eqnarray*}
  we infer
  \begin{eqnarray*}
    \nabla^{1, 0}_{g, J}  \,\bar{\partial}_{_{T_{X, J}}} B \;\; = \;\; -\;\,
    \partial^2_{r, \bar{p}} \;B_{k, \bar{l}} \;\zeta^{\ast}_r \otimes \left(
    \zeta_k^{\ast} \wedge \bar{\zeta}_p^{\ast} \right) \otimes \zeta_l \;\,+\;\,
    \tmop{Conj}\;,
  \end{eqnarray*}
  and thus using the expression (\ref{expr-adDbar}) we obtain
  \begin{eqnarray*}
    \bar{\partial}^{\ast_g}_{_{T_{X, J}}}  \,\bar{\partial}_{_{T_{X,
    J}}} B & = & - \;\,2\, \tmop{Tr}_g  \left( \nabla^{1, 0}_{g, J} \,
    \bar{\partial}_{_{T_{X, J}}} B \right)\\
    &  & \\
    & = & - \;\,4\, \nabla^{1, 0}_{g, J} \, \bar{\partial}_{_{T_{X, J}}} B
    \left( \zeta_r, \bar{\zeta}_r, \cdot \right)\\
    &  & \\
    & = & - \;\,4\, \partial^2_{r, \bar{r}} \;B_{k, \bar{l}} \;\zeta_k^{\ast} \otimes
    \zeta_l \;\,+\;\, \tmop{Conj}\;,
  \end{eqnarray*}
  at the point $p_0$. The conclusion follows from the local expression of
  $\Delta_g B$.
\end{proof}

We observe that with our conventions our adjoint operators
$$
\nabla^{\ast}_{_{T_{X, g}}}\;,\qquad \partial^{\ast_g}_{_{T_{X, J}}}\;,\qquad
   \bar{\partial}^{\ast_g}_{_{T_{X, J}}}\;, 
$$
differ by the ones usually defined in the literature by a degree
multiplicative factor. This is due to the fact that with our conventions the
metric induced on the space of forms is the restriction of the metric on the
space of tensors (without degree multiplicative factors). We refer to the
appendix in \cite{Pal1} for more details. We observe also that the formal
adjoint of the $\partial^g_{_{T_{X, J}}}$-operator with respect to the
hermitian product $\int_X \left\langle \cdot, \cdot \right\rangle_{g J}
\Omega$, is the operator
\begin{eqnarray*}
  \partial^{\ast_{g, \Omega}}_{_{T_{X, J}}} \;\; : = \;\; e^f\,
  \partial^{\ast_g}_{_{T_{X, J}}} \left( e^{- f} \bullet \right) \;.
\end{eqnarray*}
In a similar way the formal adjoint of the $\bar{\partial}_{_{T_{X,
J}}}$-operator with respect to the hermitian product $\int_X \left\langle
\cdot, \cdot \right\rangle_{g J} \Omega$, is the operator
\begin{eqnarray*}
  \bar{\partial}^{\ast_{g, \Omega}}_{_{T_{X, J}}} \;\; : = \;\; e^f\,
  \bar{\partial}^{\ast_g}_{_{T_{X, J}}} \left( e^{- f} \bullet \right) \;.
\end{eqnarray*}
With this notations hold the following lemma.

\begin{lemma}
Consider $(J, g) \in \mathcal{K}\mathcal{S}$ and let $A \in C^{\infty} (X,
  T^{\ast}_{X, - J} \otimes_{_{\mathbbm{C}}} T_{X, J}), B \in C^{\infty} (X,
  T^{\ast}_{X, J} \otimes_{_{\mathbbm{C}}} T_{X, J})$ be $g$-symmetric
  endomorphism sections. Then hold the symmetry identities
  \begin{eqnarray}
    \label{sm-Addeldel-A} \partial^{\ast_{g, \Omega}}_{_{T_{X, J}}}
    \partial^g_{_{T_{X, J}}} A &=& \left( \partial^{\ast_{g, \Omega}}_{_{T_{X,
    J}}} \partial^g_{_{T_{X, J}}} A \right)_g^T\;,
\\\nonumber
\\
    \label{com-del-AddbarOm} \partial^g_{_{T_{X, J}}}
    \bar{\partial}^{\ast_{g, \Omega}}_{_{T_{X, J}}} A &+& \frac{1}{2} \,
    \bar{\partial}^{\ast_{g, \Omega}}_{_{T_{X, J}}} \partial^g_{_{T_{X,
    J}}} A \;\;=\;\; A\, \bar{\partial}_{_{T_{X, J}}} \nabla_g f\;,
\\\nonumber
\\
  \bar{\partial}^{\ast_{g, \Omega}}_{_{T_{X, J}}} 
  \bar{\partial}_{_{T_{X, J}}} B & = & \left(
  \bar{\partial}^{\ast_{g, \Omega}}_{_{T_{X, J}}} 
  \bar{\partial}_{_{T_{X, J}}} B \right)_g^T\nonumber
\\\nonumber
  &  & \\
  & - & 2 \,\nabla_g f \;\neg\; \left( \nabla^{1, 0}_{g, J} B \;\,-\;\, \nabla^{0, 1}_{g,
  J} B \right)\nonumber
\\\nonumber
  &  & \\
  & + & 2 \left[ \tmop{Ric}^{\ast}_g, B\right]\;, 
  \label{sm-Addbardbar} 
\\\nonumber
\\
\label{com-dbar-AddelOm}  \bar{\partial}_{_{T_{X, J}}}
  \partial^{\ast_{g, \Omega}}_{_{T_{X, J}}} B &+& \frac{1}{2}\, \partial^{\ast_{g,
  \Omega}}_{_{T_{X, J}}} \bar{\partial}_{_{T_{X, J}}} B \;\;= \;\;B\,
  \bar{\partial}_{_{T_{X, J}}} \nabla_g f \;.
\end{eqnarray}
\end{lemma}
\begin{proof}
Using the definition of the adjoint operator $\partial^{\ast_{g, \Omega}}_{_{T_{X, J}}}$ and the expression (\ref{expr-adDel}) 
we expand the therm
  \begin{eqnarray*}
    \partial^{\ast_{g, \Omega}}_{_{T_{X, J}}} \partial^g_{_{T_{X, J}}} A & = &
    - \;\,2\, e^f \tmop{Tr}_g  \left[ \nabla^{0, 1}_{g, J}  \left( e^{- f}
    \partial^g_{_{T_{X, J}}} A \right) \right] \\
    &  & \\
    & = & \partial^{\ast_g}_{_{T_{X, J}}} \partial^g_{_{T_{X, J}}} A \;\,+\;\, 2
    \tmop{Tr}_g  \left( \bar{\partial}_{_J} f \otimes_{_J}
    \partial^g_{_{T_{X, J}}} A \right)\\
    &  & \\
    & = & \partial^{\ast_g}_{_{T_{X, J}}} \partial^g_{_{T_{X, J}}} A \;\,+\;\, 2\,
    \nabla_g f \;\neg\; \nabla^{1, 0}_{g, J}\, A \;.
  \end{eqnarray*}
Moreover the identity (\ref{weitz-A}) in lemma \ref{Cx-Weitz} implies the equality
  \begin{eqnarray*}
    \partial^{\ast_g}_{_{T_{X, J}}} \partial^g_{_{T_{X, J}}} A \;\; = \;\;
   \left( \partial^{\ast_g}_{_{T_{X, J}}} \partial^g_{_{T_{X, J}}} A\right)_g^T \;.
  \end{eqnarray*}
  This combined with the identity (\ref{cx-sm-A1}) implies the required
  identity (\ref{sm-Addeldel-A}). Using the definition of the operator $
  \bar{\partial}^{\ast_{g, \Omega}}_{_{T_{X, J}}}$ and the expression (\ref{expr-adDbar}) we infer
  \begin{eqnarray*}
    \bar{\partial}^{\ast_{g, \Omega}}_{_{T_{X, J}}} \partial^g_{_{T_{X,
    J}}} A & = & - \;\,2\, e^f \tmop{Tr}_g  \left[ \nabla^{1, 0}_{g, J}  \left(
    e^{- f} \partial^g_{_{T_{X, J}}} A \right) \right] \\
    &  & \\
    & = & \bar{\partial}^{\ast_g}_{_{T_{X, J}}} \partial^g_{_{T_{X, J}}}
    A + 2 \tmop{Tr}_g  \left( \partial_{_J} f \otimes_{_J} \partial^g_{_{T_{X,
    J}}} A \right)\\
    &  & \\
    & = & \bar{\partial}^{\ast_g}_{_{T_{X, J}}} \partial^g_{_{T_{X, J}}}
    A \;\,-\;\, 2\, \nabla^{1, 0}_{g, J} \,A\, \nabla_g f \;.
  \end{eqnarray*}
  We observe in fact that at the center $p_0$ of any $J$-holomorphic geodesic coordinates hold the
  following equalities
  \begin{eqnarray*}
    \tmop{Tr}_g \left( \partial_{_J} f \otimes_{_J} \partial^g_{_{T_{X, J}}} A
    \right) & = & 2\, f_l  \left( \bar{\zeta}_l \;\neg_{_J} \partial^g_{_{T_{X,
    J}}} A \right)\\
    &  & \\
    & = & - \;\,2\, f_l\, \partial_p \,C_{k, \bar{l}} \;\zeta^{\ast}_p \otimes \zeta_k \;\,+\;\,
    \tmop{Conj}\\
    &  & \\
    & = & - \;\,\nabla^{1, 0}_{g, J} \,A\, \nabla_g f\; .
  \end{eqnarray*}
  On the other hand using the definition of $\bar{\partial}^{\ast_{g, \Omega}}_{_{T_{X, J}}}$ and 
the expression (\ref{expr-adDbar}) we obtain
  \begin{eqnarray*}
    \bar{\partial}^{\ast_{g, \Omega}}_{_{T_{X, J}}} A & = & - \;\,e^f
    \tmop{Tr}_g  \left[ \nabla^{1, 0}_{g, J}  \left( e^{- f} A \right)
    \right]\\
    &  & \\
    & = & \bar{\partial}^{\ast_g}_{_{T_{X, J}}} A \;\,+\;\, \tmop{Tr}_g  \left(
    \partial_{_J} f \otimes_{_J} A \right)\\
    &  & \\
    & = & \bar{\partial}^{\ast_g}_{_{T_{X, J}}} A \;\,+\;\, A\, \nabla_g f\;,
  \end{eqnarray*}
and thus
  \begin{eqnarray*}
    \partial^g_{_{T_{X, J}}} \bar{\partial}^{\ast_{g, \Omega}}_{_{T_{X,
    J}}} A & = & \partial^g_{_{T_{X, J}}}
    \bar{\partial}^{\ast_g}_{_{T_{X, J}}} A \;\,+\;\, \partial^g_{_{T_{X, J}}}
    \left( A \,\nabla_g f \right)\\
    &  & \\
    & = & \partial^g_{_{T_{X, J}}} \bar{\partial}^{\ast_g}_{_{T_{X, J}}}
    A \;\,+\;\, \nabla^{1, 0}_{g, J} \,A\, \nabla_g f \;\,+\;\, A\, \bar{\partial}_{_{T_{X,
    J}}} \,\nabla_g f \;.
\end{eqnarray*}
In fact
\begin{eqnarray*}
    2 \,\partial^g_{_{T_{X, J}}} \left( A \,\nabla_g f \right) & = & \nabla_g 
    \left( A \,\nabla_g f \right) \;\,-\;\, J\, \nabla_{g, J \cdot}  \left( A \,\nabla_g f
    \right)\\
    &  & \\
    & = & \nabla_g A \,\nabla_g f \;\,+\;\, A\, \nabla^2_g f 
\\
\\
&-& J\, \nabla_{g, J \cdot} \,A\,
    \nabla_g f \;\,-\;\, J\, A \,\nabla_{g, J \cdot} \nabla_g f\\
    &  & \\
    & = & 2\, \nabla^{1, 0}_{g, J} \,A\, \nabla_g f \;\,+\;\, A \left( \nabla^2_g f \;\,+\;\, J\,
    \nabla_{g, J \cdot} \nabla_g f \right)\\
    &  & \\
    & = & 2\, \nabla^{1, 0}_{g, J} \,A\, \nabla_g f \;\,+\;\, 2\, A\,
    \bar{\partial}_{_{T_{X, J}}} \nabla_g f \;.
\end{eqnarray*}
We observe now that in degree $1$ hold the identity
\begin{equation}
    \label{comut-del-addelb} \partial^g_{_{T_{X, J}}}
    \bar{\partial}^{\ast_g}_{_{T_{X, J}}} \;\,+\;\, \frac{1}{2} \,
    \bar{\partial}^{\ast_g}_{_{T_{X, J}}} \partial^g_{_{T_{X, J}}} \;\;=\;\; J
    \left[ \partial^g_{_{T_{X, J}}}, \left[ \omega^{\ast}, \partial^g_{_{T_{X,
    J}}} \right] \right] \;\;=\;\; 0\; .
\end{equation}
In fact the first equality in (\ref{comut-del-addelb}) follows from a
  standard K\"ahler identity. The last equality in (\ref{comut-del-addelb})
  follows from the graded Jacobi identity and the identity
  \begin{eqnarray*}
    \left( \partial^g_{_{T_{X, J}}} \right)^2 \;\; = \;\; 0\; .
  \end{eqnarray*}
  Combining the previous formulas we infer the identity
  (\ref{com-del-AddbarOm}). Using the definition of the adjoint operator $\bar{\partial}^{\ast_{g, \Omega}}_{_{T_{X, J}}}$ and 
the expression (\ref{expr-adDbar}) we expand the therm
  \begin{eqnarray*}
    \bar{\partial}^{\ast_{g, \Omega}}_{_{T_{X, J}}} 
    \bar{\partial}_{_{T_{X, J}}} B & = & - \;\,2\, e^f \tmop{Tr}_g  \left[
    \nabla^{1, 0}_{g, J}  \left( e^{- f}  \bar{\partial}_{_{T_{X, J}}}
    B \right) \right]\\
    &  & \\
    & = & \bar{\partial}^{\ast_g}_{_{T_{X, J}}} 
    \bar{\partial}_{_{T_{X, J}}} B \;\,+\;\, 2 \tmop{Tr}_g  \left( \partial_{_J}
    f \otimes_{_J} \bar{\partial}_{_{T_{X, J}}} B \right) \\
    &  & \\
    & = & \bar{\partial}^{\ast_g}_{_{T_{X, J}}} 
    \bar{\partial}_{_{T_{X, J}}} B \;\,+\;\, 2\, \nabla_g f \;\neg\; \nabla^{0, 1}_{g,
    J} \,B\; .
  \end{eqnarray*}
  In fact at the point $p_0$ hold the
  following equalities
  \begin{eqnarray*}
    \tmop{Tr}_g  \left( \partial_{_J} f \otimes_{_J}
    \bar{\partial}_{_{T_{X, J}}} B \right) & = & 2\, f_r  \left(
    \bar{\zeta}_r \;\neg_{_J}  \bar{\partial}_{_{T_{X, J}}} B \right)\\
    &  & \\
    & = & 2\, f_r \partial_{\bar{r}} \;B_{k, \bar{l}} \;\zeta^{\ast}_k \otimes
    \zeta_l \;\,+\;\, \tmop{Conj}\\
    &  & \\
    & = & \nabla_g f \;\neg\; \nabla^{0, 1}_{g, J} \,B\; .
  \end{eqnarray*}
  Then identity (\ref{sm-Addbardbar}) follows combining lemma \ref{Cx-Weitz}
  with the identity identity
  \begin{equation}
    \label{Kh-sm-B1} \xi \;\neg\; \nabla_{g, J}^{1, 0} \,B \;\;=\;\; \left( \xi \;\neg\;
    \nabla_{g, J}^{0, 1} \,B \right)_g^T \;.
  \end{equation}
  This last follows immediately from the symmetry identity $B = B_g^T$. Using the
  definition of the operator $\partial^{\ast_{g, \Omega}}_{_{T_{X, J}}}$ and (\ref{expr-adDel}) we
  infer the equalities
  \begin{eqnarray*}
    \partial^{\ast_{g, \Omega}}_{_{T_{X, J}}} \bar{\partial}_{_{T_{X,
    J}}} B & = & -\;\, 2\, e^f \tmop{Tr}_g  \left[ \nabla^{0, 1}_{g, J}  \left(
    e^{- f}  \bar{\partial}_{_{T_{X, J}}} B \right) \right]\\
    &  & \\
    & = & \partial^{\ast_g}_{_{T_{X, J}}} \bar{\partial}_{_{T_{X, J}}} B
    \;\,+\;\, 2 \tmop{Tr}_g  \left( \bar{\partial}_{_J} f \otimes_{_J}
    \bar{\partial}_{_{T_{X, J}}} B \right)\\
    &  & \\
    & = & \partial^{\ast_g}_{_{T_{X, J}}} \bar{\partial}_{_{T_{X, J}}} B
    \;\,-\;\, 2\, \nabla^{0, 1}_{g, J} \,B\, \nabla_g f\; .
  \end{eqnarray*}
  Indeed at the point $p_0$ hold the
  equalities
  \begin{eqnarray*}
    \tmop{Tr}_g  \left( \bar{\partial}_{_J} f \otimes_{_J}
    \bar{\partial}_{_{T_{X, J}}} B \right) & = & 2\, f_{\bar{k}}  \left(
    \zeta_k \;\neg_{_J}  \bar{\partial}_{_{T_{X, J}}} B \right)\\
    &  & \\
    & = & -\;\, 2\, f_{\bar{k}} \;\partial_{\bar{p}} \;B_{k, \bar{l}} \;
    \bar{\zeta}^{\ast}_p \otimes \zeta_l \;\,+\;\, \tmop{Conj}\\
    &  & \\
    & = & - \;\,\nabla^{0, 1}_{g, J} \,B\, \nabla_g f \;.
  \end{eqnarray*}
  On the other hand combining the definition of $\partial^{\ast_{g, \Omega}}_{_{T_{X, J}}}$ and (\ref{expr-adDel}) we obtain
  \begin{eqnarray*}
    \partial^{\ast_{g, \Omega}}_{_{T_{X, J}}} B & = & - \;\,e^f \tmop{Tr}_g 
    \left[ \nabla^{0, 1}_{g, J}  \left( e^{- f} B \right) \right]\\
    &  & \\
    & = & \partial^{\ast_g}_{_{T_{X, J}}} B \;\,+\;\, \tmop{Tr}_g  \left(
    \bar{\partial}_{_J} f \otimes_{_J} B \right)\\
    &  & \\
    & = & \partial^{\ast_g}_{_{T_{X, J}}} B \;\,+\;\, B\, \nabla_g f\;,
  \end{eqnarray*}
  and thus
  \begin{eqnarray*}
    \bar{\partial}_{_{T_{X, J}}} \partial^{\ast_{g, \Omega}}_{_{T_{X,
    J}}} B & = & \bar{\partial}_{_{T_{X, J}}} \partial^{\ast_g}_{_{T_{X,
    J}}} B \;\,+\;\, \bar{\partial}_{_{T_{X, J}}} \left( B \,\nabla_g f \right)\\
    &  & \\
    & = & \bar{\partial}_{_{T_{X, J}}} \partial^{\ast_g}_{_{T_{X, J}}} B
    \;\,+\;\, \nabla^{0, 1}_{g, J} \,B\, \nabla_g f \;\,+\;\, B\, \bar{\partial}_{_{T_{X, J}}}
    \nabla_g f \;.
  \end{eqnarray*}
  In fact
  \begin{eqnarray*}
    2 \,\bar{\partial}_{_{T_{X, J}}}  \left( B \,\nabla_g f \right) & = &
    \nabla_g  \left( B \,\nabla_g f \right) \;\,+\;\, J\, \nabla_{g, J \cdot}  \left( B\,
    \nabla_g f \right)\\
    &  & \\
    & = & \nabla_g B \,\nabla_g f \;\,+\;\, B\, \nabla^2_g f 
\\
\\
&+& J\, \nabla_{g, J \cdot} \,B\,
    \nabla_g f \;\,+\;\, J\, \,B \nabla_{g, J \cdot} \nabla_g f\\
    &  & \\
    & = & 2 \,\nabla^{0, 1}_{g, J} \,B\, \nabla_g f \;\,+\;\, B \left( \nabla^2_g f \;\,+\;\, J\,
    \nabla_{g, J \cdot} \nabla_g f \right)\\
    &  & \\
    & = & 2 \,\nabla^{0,1}_{g, J} \,B\, \nabla_g f \;\,+\;\, 2\, B\,
    \bar{\partial}_{_{T_{X, J}}} \nabla_g f \;.
  \end{eqnarray*}
  Then (\ref{com-dbar-AddelOm}) follows combining the previous identities with
  the analogue of (\ref{comut-del-addelb}) in degree 1
\begin{eqnarray}
\label{comut-bar-addel}  \bar{\partial}_{_{T_{X, J}}}
     \partial^{\ast_g}_{_{T_{X, J}}} \;\,+\;\, \frac{1}{2}\, \partial^{\ast_g}_{_{T_{X,
     J}}} \bar{\partial}_{_{T_{X, J}}} \;\;=\;\; J \left[
     \bar{\partial}_{_{T_{X, J}}}, \left[ \omega^{\ast},
     \bar{\partial}_{_{T_{X, J}}} \right] \right] \;\;=\;\; 0\; . 
\end{eqnarray}
\end{proof}

\section{ Variation formulas for the complex components of the
$\Omega$-Bakry-Emery-Ricci
endomorphism}

In this section and in the next one we will use in a crucial way the symmetry
properties of the variations of K\"ahler structures obtained in \cite{Pal1}. For
this reason we will remind here a few key facts obtained in \cite{Pal1}. First of
all we remind (see lemma 3, identity (16) in \cite{Pal1}) that a smooth path
$(g_t, J_t)_t \subset \mathcal{K}\mathcal{S}$ satisfies $\dot{J}_t = (
\dot{J}_t)_{g_t}^T$ iff the family $(J_t)_{t \geqslant 0}$ is solution of the
$\tmop{ODE}$
$$
 2 \,\dot{J}_t \;\;=\;\; J_t\,  \dot{g}_t^{\ast} \;\,- \;\,
   \dot{g}_t^{\ast} J_t \;. 
$$
We remind also the following basic fact obtained in \cite{Pal1}.

\begin{lemma}
  Let $(g_t)_{t \geqslant 0}$ be an arbitrary smooth family of Riemannian
  metrics and let $(J_t)_{t \geqslant 0}$ be a family of endomorphisms of
  $T_X$ solution of the $\tmop{ODE}$
$$
 2\, \dot{J}_t \;\;=\;\; J_t \, \dot{g}_t^{\ast} \;\,- \;\,
     \dot{g}_t^{\ast} J_t\;, 
$$
  with initial conditions $J^2_0 \; = \; - \;\I_{T_X}$
  and $(J_0)^T_{g_0} \; = \; -\; J_0$. Then this
  conditions are preserved in time i.e. $J^2_t\; =\; -\; \I_{T_X}$ and $(J_t)^T_{g_t} \; =\; -\; J_t$ for all $t \geqslant 0$. 
\end{lemma}

For any $J \in \mathcal{J}$ and any $J$-invariant $g \in \mathcal{M}$ we
define in \cite{Pal1} the vector space
\begin{eqnarray*}
\mathbbm{D}^J_g \;\; \assign \;\; \Big\{ v \in C^{\infty} \left( X,
  S_{_{\mathbbm{R}}}^2 T^{\ast}_X \right) \mid \hspace{0.25em} \nabla_{g,
  J}^{0, 1} \,v_g^{\ast} \cdot \xi \;\;=\;\; \left( \nabla_{g, J}^{0, 1} \,v_g^{\ast}
  \cdot \xi \right)_g^T\;, \;\;\forall \xi \;\in\; T_X \Big\} \;.
\end{eqnarray*}
With this notation hold the following important lemma obtained in \cite{Pal1}.

\begin{lemma}
  \label{Kah-crv}Let $(g_t)_{t \geqslant 0}$ be an arbitrary smooth family of
  Riemannian metrics and let $(J_t)_{t \geqslant 0}$ be a family of
  endomorphisms of $T_X$ solution of the $\tmop{ODE}$
$$
2\, \dot{J}_t \;\;=\;\; J_t\,  \dot{g}_t^{\ast} \;\,- \;\,
     \dot{g}_t^{\ast} J_t\;, 
$$
  with K\"ahler initial data $(J_0, g_0)$. Then $(J_t, g_t)_{t \geqslant 0}$
  is a smooth family of K\"ahler structures if and only if $ \dot{g}_t \in
  \mathbbm{D}^{J_t}_{g_t}$  for all $t \geqslant 0$.
\end{lemma}

Moreover (see \cite{Pal1}) for any K\"ahler structure $(J, g)$ hold the identity
\begin{eqnarray}
  \label{Kah-D} \mathbbm{D}^J_g \;\,=\;\, \Big\{ v \in C^{\infty} \left( X,
  S_{_{\mathbbm{R}}}^2 T^{\ast}_X \right) \mid \hspace{0.25em}
  \partial^g_{_{T_{X, J}}} (v_g^{\ast})_{_J}^{1, 0} \;\,=\;\, 0\;,\;\;
  \bar{\partial}_{_{T_{X, J}}} (v_g^{\ast})_{_J}^{0, 1} \;\,=\;\, 0 \Big\}\;.\,
\end{eqnarray}
We show now our general first variation formulas for the complex components of
the $\Omega$-Bakry-Emery-Ricci endomorphism.

\begin{theorem}
  Let $(g_t, J_t)_t \subset \mathcal{K}\mathcal{S}
  $ be a smooth path such that $\dot{J}_t = ( \dot{J}_t)_{g_t}^T$. Let
  also $f_t \assign \log \frac{d V_{g_t}}{\Omega}$, with $\Omega > 0$ a smooth
  volume form. Then hold the following first variation formulas for the
  complex components of the $\Omega$-Bakry-Emery-Ricci endomorphism;
  \begin{eqnarray}
    2\, \frac{d}{dt} \tmop{Ric}^{\ast}_{_{J_t}} (\Omega)_{g_t} & = & -\;\,
    \partial^{g_t}_{_{T_X, J_t}} \bar{\partial}^{\ast_{g_t,
    \Omega}}_{_{T_{X, J_t}}}  \dot{g}_t^{0, 1} \;\,-\;\, \left( \partial^{g_t}_{_{T_X,
    J_t}} \bar{\partial}^{\ast_{g_t, \Omega}}_{_{T_{X, J_t}}} \nonumber
    \dot{g}_t^{0, 1} \right)_t^T\\\nonumber
    &  & \\
    & + & \left[ \tmop{Ric}^{\ast}_{_{J_t}} (\Omega)_{g_t}, \dot{g}_t^{0, 1}
    \right] \;\,-\;\, 2\, \dot{g}_t^{1, 0} \tmop{Ric}^{\ast}_{_{J_t}} (\Omega)_{g_t}\;,\label{var-EndOm-RcCx}
  \end{eqnarray}
  \begin{eqnarray}
    2 \,\frac{d}{dt}  \left( \bar{\partial}_{_{T_{X, J_t}}} \nabla_{g_t}
    f_t  \right) & = & - \;\,\bar{\partial}_{_{T_{X, J_t}}}
    \partial^{\ast_{g_t, \Omega}}_{_{T_{X, J_t}}}  \dot{g}_t^{1, 0} \;\,-\;\, \left(
    \bar{\partial}_{_{T_{X, J_t}}} \partial^{\ast_{g_t, \Omega}}_{_{T_{X,
    J_t}}}  \dot{g}_t^{1, 0} \right)_t^T\nonumber
\\\nonumber
    &  & \\
    & - & \left( J_t \nabla_{g_t} f_t \right) \;\neg\; \left( J_t \nabla_{g_t} 
    \dot{g}_t^{0, 1} \right)\nonumber
\\\nonumber
    &  & \\
    & + & \left[ \bar{\partial}_{_{T_{X, J_t}}}\nabla_{g_t} f_t\,,
    \dot{g}_t^{\ast} \right] \nonumber
\\\nonumber
\\
&-& \dot{g}_t^{0, 1} \,\partial^{g_t}_{_{T_X, J_t}}
    \nabla_{g_t} f_t \;\,-\;\, \partial^{g_t}_{_{T_X, J_t}} \nabla_{g_t} f_t 
    \,\dot{g}_t^{0, 1}\;,\label{var-II-cmpx}
  \end{eqnarray}
  with $\dot{g}_t^{1, 0} \assign ( \dot{g}_t^{\ast})_{_{J_t}}^{1, 0}$ and
  $\dot{g}_t^{0, 1} \assign ( \dot{g}_t^{\ast})_{_{J_t}}^{0, 1} .$
\end{theorem}
\begin{proof}
  For notation simplicity we set $A_t \;\assign\; -\; \dot{g}_t^{0, 1} \;=\; J_t
  \dot{J}_t$ and $B_t \;\assign\; \dot{g}_t^{1, 0}$. Time deriving the trivial
  identity
  \begin{eqnarray*}
    \tmop{Ric}^{\ast}_{_{J_t}} (\Omega)_{g_t} & = & - J_t
    \tmop{Ric}^{\ast}_{_{J_t}} (\Omega)_{g_t} J_t 
\\
\\
&=& -\;\, J_t \,\omega^{- 1}_t
    \tmop{Ric}_{_{J_t}} (\Omega) J_t 
\\
\\
&=& -\;\, g^{- 1}_t \tmop{Ric}_{_{J_t}}
    (\Omega) J_t\;,
  \end{eqnarray*}
  we obtain the expression
  \begin{eqnarray*}
    2 \,\frac{d}{dt} \tmop{Ric}^{\ast}_{_{J_t}} (\Omega)_{g_t} & = & 2\,
    \dot{g}_t^{\ast} g^{- 1}_t \tmop{Ric}_{_{J_t}} (\Omega) J_t \;\,-\;\, 2\, g^{- 1}_t \,
    \frac{d}{d t}  \left( \tmop{Ric}_{_{J_t}} (\Omega)  J_t \right)\\
    &  & \\
    & = & - \;\,2\, \dot{g}_t^{\ast} \tmop{Ric}^{\ast}_{_{J_t}} (\Omega)_{g_t} \;\,+\;\,
    \left( \tmop{div}^{^{_{_{\Omega}}}}_{g_t} \mathcal{D}_{g_t}  \dot{g}''_t
    \right)^{\ast}_t \\
    &  & \\
    & + & \nabla_{g_t} f_t \;\neg\; \left( \nabla^{1, 0}_{g_t, J_t} A_t \;\,-\;\,
    \nabla^{0, 1}_{g_t, J_t} A_t \right)\\
    &  & \\
    & + & \bar{\partial}_{_{T_{X, J_t}}} \nabla_{g_t} f_t \,A_t \;\,+\;\, A_t \,
    \bar{\partial}_{_{T_{X, J_t}}} \nabla_{g_t} f_t \\
    &  & \\
    & - & (i\, \partial_{_{J_t}} \bar{\partial}_{_{J_t}} f_t)_t^{\ast} \,A_t
    \;\,-\;\, A_t \,(i\, \partial_{_{J_t}} \bar{\partial}_{_{J_t}} f_t)_t^{\ast}\;,
  \end{eqnarray*}
  thanks to corollary \ref{rm-vr-Rc-fm} and thanks to the identity
  \begin{eqnarray*}
    \partial^{g_t}_{_{T_X, J_t}} \nabla_{g_t} f_t \;\; = \;\; (i \partial_{_{J_t}}
    \bar{\partial}_{_{J_t}} f_t)_t^{\ast}\;,
  \end{eqnarray*}
  which follows from the decomposition formula of the Hessian
  (\ref{cx-dec-Hess}). Using lemmas \ref{Endo-Div}, \ref{Kah-crv} and the
  equality (\ref{Kah-D}) we expand the therm
\begin{eqnarray*}
       \left( \tmop{div}^{^{_{_{\Omega}}}}_{g_t} \mathcal{D}_{g_t} 
       \dot{g}''_t \right)^{\ast}_t & = & - \;\,\frac{1}{2}\,
       \nabla^{\ast_{_{\Omega}}}_{_{T_X, g_t}} \partial^{g_t}_{_{T_X, J_t}}
       A_t \;\,-\;\, \frac{1}{2}\,  \left( \nabla^{\ast_{_{\Omega}}}_{_{T_X, g_t}}
       \partial^{g_t}_{_{T_X, J_t}} A_t \right)_t^T \;\,+\;\,
       \Delta^{^{_{_{\Omega}}}}_{g_t} A_t\\
       &  & \\
       & = & - \;\,\partial^{\ast_{g_t, \Omega}}_{_{T_{X, J_t}}} \partial^{g_t
       }_{_{T_{X, J_t}}} A_t \;\,+\;\, \Delta_{g_t} A_t \;\,+\;\, \nabla_{g_t} f_t \;\neg\;
       \nabla_{g_t} A_t\\
       &  & \\
       & - & \frac{1}{2}\,  \bar{\partial}^{\ast_{g_t, \Omega}}_{_{T_{X,
       J_t}}} \partial^{g_t}_{_{T_X, J_t}} A_t \;\,-\;\, \frac{1}{2}  \left(
       \bar{\partial}^{\ast_{g_t, \Omega}}_{_{T_{X, J_t}}}
       \partial^{g_t}_{_{T_X, J_t}} A_t \right)_t^T\;,
\end{eqnarray*}
  thanks to formula (\ref{sm-Addeldel-A}). Using the identity (\ref{weitz-A})
  we obtain
\begin{eqnarray*} 
       \left( \tmop{div}^{^{_{_{\Omega}}}}_{g_t} \mathcal{D}_{g_t} 
       \dot{g}''_t \right)^{\ast}_t & = & -\;\, A_t\, \tmop{Ric}_{g_t}^{\ast} \;\,-\;\,
       \tmop{Ric}_{g_t}^{\ast} \,A_t\\
       &  & \\
       & - & \nabla_{g_t} f_t \;\neg\, \left( \nabla^{1, 0}_{g_t, J_t} A_t \;\,-\;\,
       \nabla^{0, 1}_{g_t, J_t} A_t \right)\\
       &  & \\
       & - & \frac{1}{2}  \,\bar{\partial}^{\ast_{g_t, \Omega}}_{_{T_{X,
       J_t}}} \partial^{g_t}_{_{T_X, J_t}} A_t \;\,-\;\, \frac{1}{2}  \left(
       \bar{\partial}^{\ast_{g_t, \Omega}}_{_{T_{X, J_t}}}
       \partial^{g_t}_{_{T_X, J_t}} A_t \right)_t^T\;,
\end{eqnarray*}     
  Using formula (\ref{com-del-AddbarOm}) we deduce
  \begin{eqnarray*}
    \left( \tmop{div}^{^{_{_{\Omega}}}}_{g_t} \mathcal{D}_{g_t}  \dot{g}''_t
    \right)^{\ast}_t & = & -\;\, A_t \tmop{Ric}_{g_t}^{\ast} \;\,-\;\,
    \tmop{Ric}_{g_t}^{\ast} A_t\\
    &  & \\
    & - & \nabla_{g_t} f_t \;\neg\; \left( \nabla^{1, 0}_{g_t, J_t} A_t \;\,-\;\,
    \nabla^{0, 1}_{g_t, J_t} A_t \right)\\
    &  & \\
    & + & \partial^{g_t}_{_{T_X, J_t}} \bar{\partial}^{\ast_{g_t,
    \Omega}}_{_{T_{X, J_t}}} A_t \;\,+\;\, \left( \partial^{g_t}_{_{T_X, J_t}}
    \bar{\partial}^{\ast_{g_t, \Omega}}_{_{T_{X, J_t}}} A_t \right)_t^T
    \\
    &  & \\
    & - & A_t \,\bar{\partial}_{_{T_{X, J_t}}} \nabla_{g_t} f_t \;\,-\;\,
    \bar{\partial}_{_{T_{X, J_t}}} \nabla_{g_t} f_t \,A_t \;.
  \end{eqnarray*}
  Then the variation formula (\ref{var-EndOm-RcCx}) follows by plunging this
  expression in the previous formula for the variation of
  $\tmop{Ric}^{\ast}_{_{J_t}} (\Omega)$ and rearranging the leading therms. We
  show now formula (\ref{var-II-cmpx}). By corollary \ref{vr-ant-Hess} we
  infer the expression
  \begin{eqnarray*}
    2 \,\frac{d}{dt}  \left( \bar{\partial}_{_{T_{X, J_t}}} \nabla_{g_t}
    f_t  \right) & = & - \;\,2\, \dot{g}_t^{\ast}\, \bar{\partial}_{_{T_{X,
    J_t}}} \nabla_{g_t} f_t \;\,+\;\, \left( \tmop{div}^{^{_{_{\Omega}}}}_{g_t}
    \mathcal{D}_{g_t}  \dot{g}'_t \right)^{\ast}_t \\
    &  & \\
    & - & \nabla_{g_t} f_t \;\neg\; \left( \nabla^{1, 0}_{g_t, J_t} \,A_t \;\,-\;\,
    \nabla^{0, 1}_{g_t, J_t} \,A_t \right)\\
    &  & \\
    & - & \bar{\partial}_{_{T_{X, J_t}}} \nabla_{g_t} f_t \,A_t \;\,-\;\, A_t\, 
    \bar{\partial}_{_{T_{X, J_t}}} \nabla_{g_t} f_t\\
    &  & \\
    & + & \partial^{g_t}_{_{T_X, J_t}}\nabla_{g_t} f_t \,A_t \;\,+\;\, A_t\,
    \partial^{g_t}_{_{T_X, J_t}} \nabla_{g_t} f_t \;.
  \end{eqnarray*}
  Using again lemmas \ref{Endo-Div}, \ref{Kah-crv} and the equality
  (\ref{Kah-D}) we expand the therm
  \begin{eqnarray*}
    \left( \tmop{div}^{^{_{_{\Omega}}}}_{g_t} \mathcal{D}_{g_t}  \dot{g}'_t
    \right)^{\ast}_t & = & \frac{1}{2}\, \nabla^{\ast_{_{\Omega}}}_{_{T_X, g_t}}
    \bar{\partial}_{_{T_{X, J_t}}} B_t \;\,+\;\, \frac{1}{2}  \left(
    \nabla^{\ast_{_{\Omega}}}_{_{T_X, g_t}} \bar{\partial}_{_{T_{X,
    J_t}}} B_t \right)_t^T \;\,-\;\, \Delta^{^{_{_{\Omega}}}}_{g_t} B_t\\
    &  & \\
    & = & \frac{1}{2}\, \partial^{\ast_{g_t, \Omega}}_{_{T_{X, J_t}}}
    \bar{\partial}_{_{T_{X, J_t}}} B_t \;\,+\;\, \frac{1}{2}  \left(
    \partial^{\ast_{g_t, \Omega}}_{_{T_{X, J_t}}} \bar{\partial}_{_{T_{X,
    J_t}}} B_t \right)_t^T\\
    &  & \\
    & + & \bar{\partial}^{\ast_{g_t, \Omega}}_{_{T_{X, J_t}}}
    \bar{\partial}_{_{T_{X, J_t}}} B_t \;\,-\;\, \nabla_{g_t} f_t \;\neg\; \nabla^{0,
    1}_{g_t, J_t} B_t \;\,+\;\, \nabla_{g_t} f_t \;\neg\; \nabla^{1, 0}_{g_t, J_t} B_t\\
    &  & \\
    & - & B_t \tmop{Ric}^{\ast}_{g_t} \;\,+\;\, \tmop{Ric}^{\ast}_{g_t} B_t \;\,-\;\,
    \Delta_{g_t} B_t \;\,-\;\, \nabla_{g_t} f_t \;\neg\; \nabla_{g_t} B_t\;,
  \end{eqnarray*}
  thanks to the identity (\ref{sm-Addbardbar}). Moreover simplifying and using
  the formula (\ref{weitz-B}) we obtain the identity
  \begin{equation}
    \label{cx-weitz-B}  \left( \tmop{div}^{^{_{_{\Omega}}}}_{g_t}
    \mathcal{D}_{g_t}  \dot{g}'_t \right)^{\ast}_t \;\;=\;\; \frac{1}{2}\,
    \partial^{\ast_{g_t, \Omega}}_{_{T_{X, J_t}}} \bar{\partial}_{_{T_{X,
    J_t}}} B_t \;\,+\;\, \frac{1}{2}  \left( \partial^{\ast_{g_t, \Omega}}_{_{T_{X,
    J_t}}} \bar{\partial}_{_{T_{X, J_t}}} B_t \right)_t^T \;.
  \end{equation}
  Applying the commutation formula (\ref{com-dbar-AddelOm}) to the the
  identity (\ref{cx-weitz-B}) we infer
  \begin{eqnarray*}
    \left( \tmop{div}^{^{_{_{\Omega}}}}_{g_t} \mathcal{D}_{g_t}  \dot{g}'_t
    \right)^{\ast}_t & = & - \;\,\bar{\partial}_{_{T_{X, J_t}}}
    \partial^{\ast_{g_t, \Omega}}_{_{T_{X, J_t}}} B_t \;\,-\;\, \left(
    \bar{\partial}_{_{T_{X, J_t}}} \partial^{\ast_{g_t, \Omega}}_{_{T_{X,
    J_t}}} B_t \right)_t^T\\
    &  & \\
    & + & B_t \,\bar{\partial}_{_{T_{X, J_t}}} \nabla_{g_t} f_t \;\,+\;\,
    \bar{\partial}_{_{T_{X, J_t}}} \nabla_{g_t} f_t \,B_t\; .
  \end{eqnarray*}
  Then the variation formula (\ref{var-II-cmpx}) follows by plunging this
  expression in the previous expression of the variation of
  $\bar{\partial}_{_{T_{X, J_t}}}\nabla_{g_t} f_t$.
\end{proof}

We notice that one can deduce also the variation formula
\begin{eqnarray*}
  2\, \frac{d}{dt} \tmop{Ric}^{\ast}_{_{J_t}} (\Omega)_{g_t} & = & - \;\,\frac{1}{2}\,
  \bar{\partial}^{\ast_{g_t, \Omega}}_{_{T_{X, J_t}}}
  \partial^{g_t}_{_{T_X, J_t}} A_t \;\,-\;\, \frac{1}{2}  \left(
  \bar{\partial}^{\ast_{g_t, \Omega}}_{_{T_{X, J_t}}}
  \partial^{g_t}_{_{T_X, J_t}} A_t \right)_t^T\\
  &  & \\
  & + & A_t  \,\bar{\partial}_{_{T_{X, J_t}}} \nabla_{g_t} f_t \;\,+\;\,
  \bar{\partial}_{_{T_{X, J_t}}} \nabla_{g_t} f_t \,A_t\\
  &  & \\
  & - & \left[ \tmop{Ric}^{\ast}_{_{J_t}} (\Omega)_{g_t}, A_t \right] \;\,-\;\, 2\, B_t\,
  \tmop{Ric}^{\ast}_{_{J_t}} (\Omega)_{g_t} \;.
\end{eqnarray*}

\section{Representation of the Soliton-K\"ahler-Ricci flow as a complex
strictly parabolic system}

In this section we will show that the Soliton-K\"ahler-Ricci flow (introduced
in \cite{Pal3}) generated by the Soliton-Ricci flow (introduced in \cite{Pal2}) represents
 a strictly parabolic system of the complex components of the metric
variation. Let
\begin{eqnarray*}
  \mathbbm{F}_g \;\; \assign \;\; \left \{ v \in C^{\infty} \left( X,
  S_{_{\mathbbm{R}}}^2 T^{\ast}_X \right) \mid \hspace{0.25em} \nabla_{_{T_X,
  g}} v_g^{\ast} = 0 \right\} \;.
\end{eqnarray*}
We remind (see \cite{Pal2}) that the Soliton-Ricci flow is a strictly
parabolic equation. Indeed we observe that the variation formula
(\ref{var-Om-Ric}) implies directly that if $(g_t)_{t \in \mathbbm{R}} \subset
\mathcal{M}$ is a smooth family such that $\dot{g}_t \in \mathbbm{F}_{g_t}$
for all $t \in \mathbbm{R}$, then
\begin{eqnarray*}
  2 \,\frac{d}{d t} \tmop{Ric}_{g_t} (\Omega) \;\; = \;\; -\;\,
  \Delta^{^{_{_{\Omega}}}}_{g_t}  \dot{g}_t \;.
\end{eqnarray*}
(see also \cite{Pal2} for a different proof of this formula). We obtain in
particular the variation formula
\begin{equation}
  \label{vr-Om-EndRc} 2 \,\frac{d}{d t} \tmop{Ric}^{\ast}_{g_t} (\Omega) \;\;=\;\; -\;\,
  \Delta^{^{_{_{\Omega}}}}_{g_t}  \dot{g}^{\ast}_t \;\,-\;\, 2\, \dot{g}_t^{\ast}
  \tmop{Ric}^{\ast}_{g_t} (\Omega) \;.
\end{equation}
In \cite{Pal1} we show that for any K\"ahler structure $(J, g)$ hold the identity
\begin{equation}
  \label{Kah-F} \mathbbm{F}_g \;\;=\;\; \Big\{ v \in \mathbbm{D}^J_g \mid \hspace{0.25em} 
  \bar{\partial}_{_{T_{X, J}}} (v_g^{\ast})_{_J}^{1, 0} \;\;=\;\; -\;\,
  \partial^g_{_{T_{X, J}}} (v_g^{\ast})_{_J}^{0, 1}  \Big\}\; .
\end{equation}
It was also observed in \cite{Pal1} that comparing the previous $(1, 1)$-forms
by means of $g$-geodesic and $J$-holomorphic coordinates we can infer the
identity
\begin{equation}
  \label{sup-Kh-sm} \mathbbm{F}_g \;\;=\;\; \Big\{ v \in \mathbbm{D}^J_g \mid
  \hspace{0.25em} \xi \;\neg\; \nabla_{g, J}^{0, 1} \,(v_g^{\ast})_{_J}^{1, 0} \;\;=\;\;
  \nabla_{g, J}^{1, 0} \,(v_g^{\ast})_{_J}^{0, 1} \cdot \xi\;, \;\;\forall \;\xi \;\in\; T_X
  \Big\}\; .
\end{equation}
In order to see that the Soliton-K\"ahler-Ricci flow is a complex strictly
parabolic system we need to obtain first variation formulas for the complex
components of the $\Omega$-Bakry-Emery-Ricci endomorphism with respect to
$\mathbbm{F}_g$-valued variations of the metric. This formulas can
be obtained with little effort from the general variation formulas
(\ref{var-EndOm-RcCx}), (\ref{var-II-cmpx}) by means of the identity
(\ref{Kah-F}). However we prefer to show this particular variation formulas in
an independent way which avoids large part of the deep and heavy computations
needed to proof formulas (\ref{var-EndOm-RcCx}) and (\ref{var-II-cmpx}).

We start by showing the following quite elementary fact.

\begin{lemma}
  Let $(J, g) \in \mathcal{K}\mathcal{S}$ and let $v \in \mathbbm{F}_g$. Then
  hold the commutation identities
  \begin{eqnarray}
    \label{dadd} \partial^g_{_{T_{X, J}}} \partial^{\ast_g}_{_{T_{X, J}}}
    (v_g^{\ast})_{_J}^{1, 0} &=& \frac{1}{2} \,
    \bar{\partial}^{\ast_g}_{_{T_{X, J}}} \bar{\partial}_{_{T_{X,
    J}}} (v_g^{\ast})_{_J}^{1, 0}\;,
\\\nonumber
\\
    \label{bar-adbar}  \bar{\partial}_{_{T_{X, J}}}
    \bar{\partial}^{\ast_g}_{_{T_{X, J}}} (v_g^{\ast})_{_J}^{0, 1} &=&
    \frac{1}{2} \,\partial^{\ast_g}_{_{T_{X, J}}} \partial^g_{_{T_{X, J}}}
    (v_g^{\ast})_{_J}^{0, 1}\;.
  \end{eqnarray}
\end{lemma}
\begin{proof}
  We observe first that using the identity (\ref{Kah-F}) and the standard
  K\"ahler identities we infer
  \begin{eqnarray*}
    \partial^{\ast_g}_{_{T_{X, J}}} (v_g^{\ast})_{_J}^{1, 0} & = & J\Big[
    \omega^{\ast}\; \neg\; \bar{\partial}_{_{T_{X, J}}} (v_g^{\ast})_{_J}^{1,
    0}\Big]\\
    &  & \\
    & = & - \;\,J\Big[ \omega^{\ast}\; \neg\; \partial^g_{_{T_{X, J}}}
    (v_g^{\ast})_{_J}^{0, 1}\Big]\\
    &  & \\
    & = & \bar{\partial}^{\ast_g}_{_{T_{X, J}}} (v_g^{\ast})_{_J}^{0, 1}\;.
  \end{eqnarray*}
  Thus
  \begin{eqnarray*}
    \partial^g_{_{T_{X, J}}} \partial^{\ast_g}_{_{T_{X, J}}}
    (v_g^{\ast})_{_J}^{1, 0} & = & \partial^g_{_{T_{X, J}}}
    \bar{\partial}^{\ast_g}_{_{T_{X, J}}} (v_g^{\ast})_{_J}^{0, 1}\\
    &  & \\
    & = & - \;\,\frac{1}{2}\,  \bar{\partial}^{\ast_g}_{_{T_{X, J}}}
    \partial^g_{_{T_{X, J}}} (v_g^{\ast})_{_J}^{0, 1}\\
    &  & \\
    & = & \frac{1}{2} \, \bar{\partial}^{\ast_g}_{_{T_{X, J}}}
    \bar{\partial}_{_{T_{X, J}}} (v_g^{\ast})_{_J}^{1, 0}\;,
  \end{eqnarray*}
  thanks to the identity (\ref{comut-del-addelb}). This concludes the proof of
  the identity (\ref{dadd}). 
Moreover as before we infer the equalities
  \begin{eqnarray*}
    \bar{\partial}_{_{T_{X, J}}} \bar{\partial}^{\ast_g}_{_{T_{X,
    J}}} (v_g^{\ast})_{_J}^{0, 1} & = & \bar{\partial}_{_{T_{X, J}}}
    \partial^{\ast_g}_{_{T_{X, J}}} (v_g^{\ast})_{_J}^{1, 0}\\
    &  & \\
    & = & - \;\,\frac{1}{2}\, \partial^{\ast_g}_{_{T_{X, J}}}
    \bar{\partial}_{_{T_{X, J}}} (v_g^{\ast})_{_J}^{1, 0}\\
    &  & \\
    & = & \frac{1}{2} \,\partial^{\ast_g}_{_{T_{X, J}}} \partial^g_{_{T_{X,
    J}}} (v_g^{\ast})_{_J}^{0, 1},
  \end{eqnarray*}
  thanks to the identity (\ref{comut-bar-addel}) This concludes the proof of
  the identity (\ref{bar-adbar}). 
\end{proof}

We show now the following particular first variation formula for the
$J$-anti-linear part of the Hessian.

\begin{lemma}
  \label{vr-Hes}Let $(g_t, J_t)_t \subset
  \mathcal{K}\mathcal{S} $ be a smooth path such that $\dot{J}_t = (
  \dot{J}_t)_{g_t}^T$ and $\dot{g}_t \in \mathbbm{F}_{g_t}$. Let also $f_t
  \assign \log \frac{d V_{g_t}}{\Omega}$, with $\Omega > 0$ a smooth volume
  form. Then hold the following first variation identity for the
  $J_t$-anti-linear part of the endomorphism $\nabla^2_{g_t} f_t $ 
  \begin{eqnarray}
    2 \,\frac{d}{dt}  \left( \bar{\partial}_{_{T_{X, J_t}}}\nabla_{g_t}
    f_t  \right) & = & -\;\, 2\, \bar{\partial}_{_{T_{X, J_t}}}
    \bar{\partial}^{\ast_{g_t}}_{_{T_{X, J_t}}} \dot{g}_t^{0, 1} \;\,-\;\,
    \nabla_{g_t} f_t \;\neg\; \nabla_{g_t}  \dot{g}_t^{0, 1}\nonumber
\\\nonumber
    &  & \\
    & - & \dot{g}_t^{0, 1}\, \partial^{g_t}_{_{T_X, J_t}}\nabla_{g_t} f_t \;\,-\;\,
    \partial^{g_t}_{_{T_X, J_t}}\nabla_{g_t} f_t\,  \dot{g}_t^{0, 1}\nonumber
\\\nonumber
    &  & \\
    & + & \left[ \bar{\partial}_{_{T_{X, J_t}}}\nabla_{g_t} f_t\,,
    \dot{g}_t^{0, 1} \right] \;\,-\;\, 2\, \dot{g}_t^{1, 0}\, \bar{\partial}_{_{T_{X,
    J_t}}} \nabla_{g_t} f_t \;.\label{var-II-cmpx-F}
  \end{eqnarray}
with $\dot{g}_t^{1, 0} \assign ( \dot{g}_t^{\ast})_{_{J_t}}^{1, 0}$ and
  $\dot{g}_t^{0, 1} \assign ( \dot{g}_t^{\ast})_{_{J_t}}^{0, 1}$.
\end{lemma}
\begin{proof}
  Using the variation formula (\ref{var-dbar-vflds}) we infer the identity
  \begin{eqnarray*}
    2\, \frac{d}{dt}  \left( \bar{\partial}_{_{T_{X, J_t}}} \nabla_{g_t}
    f_t \right) & = & -\;\, \nabla_{g_t} f_t \;\neg\; J_t \nabla_{g_t}  \dot{J}_t \;\,+\;\,
    \dot{J}_t \nabla^2_{g_t} f_t \, J_t \;\, +\;\,
     J_t \nabla^2_{g_t} f_t \,  \dot{J}_t \\
    &  & \\
    & + & 2\, \bar{\partial}_{_{T_{X, J_t}}} \left( \frac{d}{dt}\,
    \nabla_{g_t} f_t \right) \;.
  \end{eqnarray*}
  Time deriving the definition of the gradient we obtain
$$
 d \dot{f}_t \;\;=\;\; \left( \frac{d}{dt}
     \,\nabla_{g_t} f_t \right) \;\neg\; g_t \;\,+\;\,\nabla_{g_t} f_t \;\neg\; \dot{g}_t \;, 
$$
  thus
$$
 \frac{d}{dt} \hspace{0.25em} \nabla_{g_t} f_t \;\;=\;\;
 \nabla_{g_t}  \dot{f}_t \;\,-\;\, 
     \dot{g}_t^{\ast} \, \nabla_{g_t} f_t \;. 
$$
  Combining this with the equalities $2 \,\dot{f}_t \; =
 \; \tmop{Tr}_{g_t}  \dot{g}_t$ and $\dot{g}_t^{0, 1} \;=\; -\; J_t
  \dot{J}_t$ we infer
  \begin{eqnarray*}
    2\, \frac{d}{dt}  \left( \bar{\partial}_{_{T_{X, J_t}}}\nabla_{g_t}
    f_t  \right) & = & \bar{\partial}_{_{T_{X, J_t}}} \nabla_{g_t}
    \tmop{Tr}_{g_t} \dot{g_t} 
\\
\\
&-& 2\, \bar{\partial}_{_{T_{X, J_t}}} \left(
    \dot{g}_t^{\ast} \, \nabla_{g_t} f_t \right) \;\,+\;\, \nabla_{g_t} f_t \;\neg\;
    \nabla_{g_t}  \dot{g}_t^{0, 1}\\
    &  & \\
    & - & \partial^{g_t}_{_{T_X, J_t}} \nabla_{g_t} f_t  \;\dot{g}_t^{0, 1} \;\,+\;\,
    \bar{\partial}_{_{T_{X, J_t}}} \nabla_{g_t} f_t \; \dot{g}_t^{0, 1} \\
    &  & \\
    & + & \dot{g}_t^{0, 1} \,\partial^{g_t}_{_{T_X, J_t}} \nabla_{g_t} f_t \;\,-\;\,
    \dot{g}_t^{0, 1}\, \bar{\partial}_{_{T_{X, J_t}}} \nabla_{g_t} f_t \;.
  \end{eqnarray*}
  Then
  \begin{eqnarray*}
    2 \,\bar{\partial}_{_{T_{X, J_t}}} \left(  \dot{g}_t^{\ast}\,
    \nabla_{g_t} f_t  \right) & = & \nabla_{g_t}  \left( \dot{g}_t^{\ast}
    \, \nabla_{g_t} f_t \right) \;\,+\;\, J_t \nabla_{g_t, J_t \cdot}  \left(
    \dot{g}_t^{\ast} \, \nabla_{g_t} f_t \right)\\
    &  & \\
    & = & \nabla_{g_t}  \dot{g}_t^{\ast} \, \nabla_{g_t} f_t \;\,+\;\,
    \dot{g}_t^{\ast} \,\nabla^2_{g_t} f_t \\
    &  & \\
    & + & J_t \nabla_{g_t, J_t \cdot}  \,\dot{g}_t^{\ast} \, \nabla_{g_t}
    f_t \;\,+\;\, J_t \, \dot{g}_t^{\ast} \nabla_{g_t, J_t \cdot} \nabla_{g_t} f_t\\
    &  & \\
    & = & 2\, \nabla_{g, J}^{0, 1}  \dot{g}_t^{\ast} \, \nabla_{g_t} f_t \;\,+\;\,
    \dot{g}_t^{1, 0} \nabla^2_{g_t} f_t \;\,+\;\, \dot{g}_t^{0, 1} \,\nabla^2_{g_t} f_t
    \\
    &  & \\
    & + & \dot{g}_t^{1, 0} J_t \nabla_{g_t, J_t \cdot} \nabla_{g_t} f_t \;\,-\;\,
    \dot{g}_t^{0, 1} J_t \nabla_{g_t, J_t \cdot} \nabla_{g_t} f_t\\
    &  & \\
    & = & 2 \,\nabla_{g_t} f_t \neg \nabla_{g_t}  \dot{g}_t^{0, 1} 
\\
\\
&+& 2\,
    \dot{g}_t^{1, 0}\, \bar{\partial}_{_{T_{X, J_t}}} \nabla_{g_t} f_t \;\,+\;\, 2\,
    \dot{g}_t^{0, 1} \,\partial^{g_t}_{_{T_X, J_t}} \nabla_{g_t} f_t\;,
  \end{eqnarray*}
  since for any vector $\xi \in T_X$ hold the identity 
$$
\nabla_{g, J}^{0, 1} 
  \,\dot{g}_t^{\ast} \cdot \xi \;\;=\;\; \xi \;\neg\; \nabla_{g_t} \, \dot{g}_t^{0, 1}\;.
$$
In
  fact decomposing and using the symmetries (\ref{sup-Kh-sm}),
  (\ref{Tg-cx-str}) we infer
  \begin{eqnarray*}
    \nabla_{g, J}^{0, 1} \, \dot{g}_t^{\ast} \cdot \xi & = & \nabla_{g, J}^{0,
    1} \, \dot{g}_t^{1, 0} \cdot \xi \;\,+\;\, \nabla_{g, J}^{0, 1} \, \dot{g}_t^{0, 1}
    \cdot \xi\\
    &  & \\
    & = & \xi \;\neg\; \nabla_{g, J}^{1, 0} \, \dot{g}_t^{0, 1} \;\,+\;\, \xi \;\neg\;
    \nabla_{g, J}^{0, 1} \, \dot{g}_t^{0, 1}\\
    &  & \\
    & = & \xi \;\neg\; \nabla_{g_t} \, \dot{g}_t^{0, 1} \;.
  \end{eqnarray*}
  We remind now (see the proof of lemma 3 in \cite{Pal2}) that the assumption
  $\dot{g}_t \in \mathbbm{F}_{g_t}$ implies the identity $\nabla_{g_t}
  \tmop{Tr}_{g_t} \dot{g_t} \;=\; -\; \nabla^{\ast}_{g_t}  \,\dot{g}_t^{\ast}$.
  Combining the previous formulas we infer the identity
  \begin{eqnarray*}
    2\, \frac{d}{dt}  \left( \bar{\partial}_{_{T_{X, J_t}}} \nabla_{g_t}
    f_t  \right) & = & - \;\,\bar{\partial}_{_{T_{X, J_t}}}
    \nabla_{g_t}^{\ast}  \,\dot{g}^{\ast}_t \;\,-\;\, \nabla_{g_t} f_t \;\neg\; \nabla_{g_t}\,
    \dot{g}_t^{0, 1} 
\\
\\
&+& \left[ \bar{\partial}_{_{T_{X, J_t}}}
    \nabla_{g_t} f_t\,, \dot{g}_t^{0, 1} \right]\\
    &  & \\
    & - & \dot{g}_t^{0, 1}\; \partial^{g_t}_{_{T_X, J_t}} \nabla_{g_t} f_t \;\,-\;\,
    \partial^{g_t}_{_{T_X, J_t}} \nabla_{g_t} f_t  \;\dot{g}_t^{0, 1} 
\\
\\
&-& 2\,
    \dot{g}_t^{1, 0} \;\bar{\partial}_{_{T_{X, J_t}}} \nabla_{g_t} f_t \;.
  \end{eqnarray*}
  The conclusion follows from the identities
  \begin{eqnarray*}
    \bar{\partial}_{_{T_{X, J_t}}} \nabla_{g_t}^{\ast}  \,\dot{g}^{\ast}_t
    & = &  \bar{\partial}_{_{T_{X, J_t}}} \partial^{\ast_{g_t}}_{_{T_{X,
    J_t}}} \dot{g}_t^{1, 0} \;\,+\;\, \bar{\partial}_{_{T_{X, J_t}}}
    \bar{\partial}^{\ast_g}_{_{T_{X, J}}}  \dot{g}_t^{0, 1}\\
    &  & \\
    & = & 2 \,\bar{\partial}_{_{T_{X, J_t}}}
    \bar{\partial}^{\ast_{g_t}}_{_{T_{X, J_t}}}  \dot{g}_t^{0, 1} \;.
  \end{eqnarray*}
\end{proof}
\begin{corollary}
  Under the assumptions of lemma \ref{vr-Hes} hold the variation formula
  \begin{eqnarray}
    2\, \frac{d}{dt} \tmop{Ric}^{\ast}_{_{J_t}} (\Omega)_{g_t} & = & -\;\, 2\,
    \partial^{g_t }_{_{T_{X, J_t}}} \partial^{\ast_{g_t}}_{_{T_{X, J_t}}} 
    \dot{g}_t^{1, 0} \;\,-\;\, \nabla_{g_t} f_t \;\neg\; \nabla_{g_t}  \,\dot{g}_t^{1, 0}\nonumber
\\\nonumber
    &  & \\
    & - & \dot{g}_t^{0, 1}\, \bar{\partial}_{_{T_{X, J_t}}} \nabla_{g_t}
    f_t \;\,-\;\, \bar{\partial}_{_{T_{X, J_t}}} \nabla_{g_t} f_t  \;\dot{g}_t^{0,
    1} \;\,-\;\, \left[ \tmop{Ric}^{\ast}_{g_t}, \dot{g}_t^{1, 0} \right]\nonumber
\\\nonumber
    &  & \\
    & + & \left[ \tmop{Ric}^{\ast}_{_{J_t}} (\Omega)_{g_t}, \dot{g}_t^{0, 1}
    \right] \;\,-\;\, 2\, \dot{g}_t^{1, 0} \tmop{Ric}^{\ast}_{_{J_t}} (\Omega)_{g_t} \;.\label{var-EndOm-RcCx-F}
  \end{eqnarray}
\end{corollary}

\begin{proof}
  Combining the commutation identity (\ref{bar-adbar}) with the formula
  (\ref{weitz-A}) in the variation formula (\ref{var-II-cmpx}) we obtain
  \begin{eqnarray*}
    2\, \frac{d}{dt}  \left( \bar{\partial}_{_{T_{X, J_t}}} \nabla_{g_t}
    f_t  \right) & = & -\;\, \Delta^{^{_{_{\Omega}}}}_{g_t} \, \dot{g}_t^{0, 1} \;\,-\;\,
    \dot{g}_t^{0, 1} \tmop{Ric}^{\ast}_{_{J_t}} (\Omega)_{g_t} \;\,-\;\,
    \tmop{Ric}^{\ast}_{_{J_t}} (\Omega)_{g_t}  \,\dot{g}_t^{0, 1} \\
    &  & \\
    & + & \left[ \bar{\partial}_{_{T_{X, J_t}}} \nabla_{g_t} f_t\,,
    \dot{g}_t^{0, 1} \right] \;\,-\;\, 2\, \dot{g}_t^{1, 0}\, \bar{\partial}_{_{T_{X,
    J_t}}} \nabla_{g_t} f_t \;.
  \end{eqnarray*}
  We remind in fact the identities \ $\tmop{Ric}^{\ast}_{g_t} =
  \tmop{Ric}^{\ast}_{_{J_t}} (\omega_t)_{g_t}$ and
  \begin{eqnarray*}
    \partial^{g_t}_{_{T_X, J_t}} \nabla_{g_t} f_t \;\; = \;\; (i\, \partial_{_{J_t}}
    \bar{\partial}_{_{J_t}} f_t)_t^{\ast} \;.
  \end{eqnarray*}
  (This last hold thanks to the decomposition formula of the Hessian
  (\ref{cx-dec-Hess}).) On the other hand time deriving the complex
  decomposition formula (\ref{dec-end-Ric}) and using the variation formula
  (\ref{vr-Om-EndRc}) we obtain
  \begin{eqnarray*}
    2\, \frac{d}{dt} \tmop{Ric}^{\ast}_{_{J_t}} (\Omega)_{g_t} & = & -\;\,
    \Delta^{^{_{_{\Omega}}}}_{g_t}  \,\dot{g}^{\ast}_t \;\,-\;\, 2\, \dot{g}^{\ast}_t
    \tmop{Ric}^{\ast}_{g_t} (\Omega) \;\,-\;\, 2\, \frac{d}{dt}  \left(
    \bar{\partial}_{_{T_{X, J_t}}} \nabla_{g_t} f_t  \right)\\
    &  & \\
    & = & - \;\,\Delta^{^{_{_{\Omega}}}}_{g_t}\,  \dot{g}_t^{1, 0} \;\,+\;\, \left[
    \tmop{Ric}^{\ast}_{_{J_t}} (\Omega)_{g_t}\,, \dot{g}_t^{0, 1} \right] \\
    &  & \\
    & - & \dot{g}_t^{0, 1} \,\bar{\partial}_{_{T_{X, J_t}}} \nabla_{g_t}
    f_t \;\,-\;\, \bar{\partial}_{_{T_{X, J_t}}} \nabla_{g_t} f_t  \;\dot{g}_t^{0,
    1} 
\\
\\
&-& 2\, \dot{g}_t^{1, 0} \tmop{Ric}^{\ast}_{_{J_t}} (\Omega)_{g_t} \;.
  \end{eqnarray*}
  Then the conclusion follows combining formula (\ref{weitz-B}) with the
  identity (\ref{dadd}).
\end{proof}

We remind that with our conventions our adjoint operators
$$
 \nabla^{\ast}_{_{T_{X, g}}}\;,\quad \partial^{\ast_g}_{_{T_{X, J}}}\;,\quad
   \bar{\partial}^{\ast_g}_{_{T_{X, J}}}\;, 
$$
differ by the ones usually defined in the literature by a degree
multiplicative factor. The motivation for our convention is to preserve
functoriality between the Riemannian and K\"ahler geometry. We remind (see \cite{Pal2}) that with our conventions the Hodge Laplacian operator acting on
$T_X$-valued $q$-forms is defined as
\begin{eqnarray*}
  \Delta_{_{T_{X, g}}} & : = & \nabla_{_{T_{X, g}}} \nabla^{\ast}_g \;\,+\;\,
  \nabla_g^{\ast} \,\nabla_{_{T_{X, g}}} \\
  &  & \\
  & = & \frac{1}{q}\, \nabla_{_{T_{X, g}}} \nabla^{\ast}_{_{T_{X, g}}} \;\,+\;\,
  \frac{1}{q + 1} \,\nabla^{\ast}_{_{T_{X, g}}} \nabla_{_{T_{X, g}}} \;.
\end{eqnarray*}
We define also the holomorphic and antiholomorphic Hodge Laplacian operators
acting on $T_X$-valued $q$-forms as
\begin{eqnarray*}
  \Delta'_{_{T_{X, g, J}}} & \assign & \frac{1}{q}\; \partial^g_{_{T_{X, J}}}
  \partial^{\ast_g}_{_{T_{X, J}}} \;\,+\;\, \frac{1}{q + 1}\; \partial^{\ast_g}_{_{T_{X,
  J}}} \partial^g_{_{T_{X, J}}}\;,\\
  &  & \\
  \Delta''_{_{T_{X, g, J}}} & \assign & \frac{1}{q} \;
  \bar{\partial}_{_{T_{X, J}}} \bar{\partial}^{\ast_g}_{_{T_{X, J}}}
  \;\,+\;\, \frac{1}{q + 1}\;  \bar{\partial}^{\ast_g}_{_{T_{X, J}}}
  \bar{\partial}_{_{T_{X, J}}}\;,
\end{eqnarray*}
with the usual convention $\infty \cdot 0 = 0.$ This complex Hodge Laplacian operators
coincide with the standard ones used in the literature. We remind that in the
K\"ahler case hold the decomposition identity
\begin{eqnarray*}
  \Delta_{_{T_{X, g}}} \;\; = \;\; \Delta'_{_{T_{X, g, J}}} \;\,+\;\, \Delta''_{_{T_{X, g,
  J}}} \;.
\end{eqnarray*}
Let $(J_t, g_t)_{t \geqslant 0}$ be a solution of the $\Omega$-Soliton-K\"ahler-Ricci
flow (see \cite{Pal2}) generated by the $\Omega$-Soliton-Ricci flow (see \cite{Pal2}). We set
as usual $A_t \;\assign\; -\; \dot{g}_t^{0, 1} \;=\; J_t \dot{J}_t$ and $B_t \;\assign\;
\dot{g}_t^{1, 0}$ and we observe the identity $\dot{\omega}_t^{\ast} \;=\; B_t$.
Indeed time differentiating the equality $\omega_t \;=\; g_t J_t$ we infer
\begin{eqnarray*}
  \dot{\omega}_t \;\; = \;\; \dot{g}_t\, J_t \;\,+\;\, g_t\,  \dot{J}_t \;\;=\;\; \dot{g}_t\, J_t \;\,+\;\,
  \omega_t \, \dot{g}_t^{0, 1} \;.
\end{eqnarray*}
Multiplying booth sides with $\omega^{- 1} \;=\; -\; J\, g^{- 1}$ we obtain
$$
\dot{\omega}_t^{\ast} \;\;=\;\; -\;\, J_t\,  \dot{g}^{\ast}_t J_t \;\,+\;\, \dot{g}_t^{0, 1} \;\,=\;\,
\dot{g}_t^{1, 0}\;.
$$
With the previous notations the equation of the $\Omega$-Soliton-K\"ahler-Ricci
flow generated by the $\Omega$-Soliton-Ricci flow rewrites as
\begin{equation}
  \label{SKRF-eq}  \left\{ \begin{array}{l}
    B_t \;\;=\;\;  \tmop{Ric}^{\ast}_{_{J_t}} (\Omega)_{g_t}
     \;\,-\;\,\mathbbm{I}_{T_X}  \;,\\
    \\
    A_t \;\;=\;\;  \bar{\partial}_{_{T_{X, J_t}}} \nabla_{g_t} f_t
    \;,\\
    \\
    \bar{\partial}_{_{T_{X, J_t}}} B \;\;=\;\; \partial_{_{T_{X, J_t}}} A \;.
  \end{array} \right.
\end{equation}
The last equation in the system follows from the equality (\ref{Kah-F}). \
Time differentiating the first two equations of the system (\ref{SKRF-eq}) by
means of the first variation formulas (\ref{var-EndOm-RcCx-F}),
(\ref{var-II-cmpx-F}) and using lemma \ref{Kah-crv} with the equality
(\ref{Kah-D}) we obtain the parabolic type evolution formulas
\begin{eqnarray}
  2\, \dot{B}_t & = & - \;\,2\, \Delta'_{_{T_{X, g_t, J_t}}} B_t \;\,-\;\, \nabla_{g_t} f_t
  \;\neg\; \nabla_{g_t} B_t\nonumber
\\\nonumber
  &  & \\
  & - & 2 \,B^2_t \;\,+\;\, 2\, A^2_t \;\,+\;\, \left[ B, A \right] \;\,-\;\, \left[
  \tmop{Ric}^{\ast}_{g_t}, B \right] \;\,-\;\, 2\, B\;,\label{evol-B}
\\\nonumber
  &  & \\
  2 \,\dot{A}_t & = & - \;\,2\, \Delta''_{_{T_{X, g_t, J_t}}} A_t \;\,-\;\, \nabla_{g_t} f_t
  \;\neg\; \nabla_{g_t} A_t\nonumber
\\\nonumber
  &  & \\
  & + & A_t \,\partial^{g_t}_{_{T_X, J_t}} \nabla_{g_t} f_t \;\,+\;\,
  \partial^{g_t}_{_{T_X, J_t}} \nabla_{g_t} f_t \,A_t \;\,-\;\, 2\, B_t\, A_t\;,\label{evol-A}
\end{eqnarray}
which show the required conclusion.

\section{Appendix}

\subsection{Canonical connections}

Over any almost complex manifold $(X, J)$ there exist a canonical connection
of type $(0, 1)$ over the vector bundle $\Lambda^p_{_{\C}} T^{\ast}_{_{X,
J}}$,
\[ \bar{\partial}_{_{J, p}} : C^{\infty} \left( \Lambda^p_{_{\C}}
   T^{\ast}_{X, J} \right) \longrightarrow C^{\infty} \left( \Lambda^{0,
   1}_{_J} T^{\ast}_X \otimes_{_{\C}} \Lambda^p_{_{\C}} T^{\ast}_{X, J}
   \right) \hspace{0.25em}, \]
given by the formula
$$ 
\bar{\partial}_{_{J, p}} \alpha \hspace{0.25em} (\eta) 
  \;\; : =\;\;  \eta^{0, 1}_{_J} \;\neg\; \bar{\partial}_{_J} \alpha
   \;,
$$
for all $\alpha \in C^{\infty} (X, \Lambda^p_{_{\C}} T^{\ast}_{X, J})$ and
$\eta \in C^{\infty} (X, T_X \otimes_{_{\R}} \C)$. More explicitly
\begin{eqnarray}
  \bar{\partial}_{_{J, p}} \alpha \hspace{0.25em} (\eta) \,(\xi_1, ...,
  \xi_p) & = & \eta^{0, 1} . \hspace{0.25em} \alpha (\xi_1, ..., \xi_p)\nonumber
\\\nonumber
  &  & \\
  & + & \sum_{l = 1}^p (- 1)^l \,\alpha ([\eta^{0, 1}, \xi_l^{1, 0}]^{1, 0},
  \xi_1, ..., \hat{\xi}_l, ..., \xi_p) \;,\label{def-dbar-p}
\end{eqnarray}
for all $\xi_j \in C^{\infty} (X, T_X)$. In order to show that
$\bar{\partial}_{_{J, p}}$ is a connection of type $(0, 1)$ we observe
that for any $f \in C^{\infty} (X, \C)$ hold the identities
\begin{eqnarray*}
  \bar{\partial}_{_{J, p}} (f \alpha) \hspace{0.25em} (\eta) & = &
  \eta^{0, 1} \;\neg\; \bar{\partial}_{_J} (f \alpha)\\
  &  & \\
  & = & \eta^{0, 1} \;\neg\; \left( \bar{\partial}_{_J} f \wedge \alpha
  \;\,+ \;\, f \hspace{0.25em} \bar{\partial}_{_J}
  \alpha \right)\\
  &  & \\
  & = & \bar{\partial}_{_J} f \hspace{0.25em} (\eta) \cdot \alpha
  \;\, - \;\, \bar{\partial}_{_J} f \wedge
  (\eta^{0, 1} \;\neg\; \alpha) \;\, + \;\, f
  \hspace{0.25em} \bar{\partial}_{_J} \alpha \hspace{0.25em} (\eta)\\
  &  & \\
  & = & \left( \bar{\partial}_{_J} f \otimes \alpha \;\, +\;\,
   f \hspace{0.25em} \bar{\partial}_{_J} \alpha \right)
  (\eta) \;,
\end{eqnarray*}
since $\alpha$ is of type $(p, 0)$. In a similar way there exists a canonical
connection of type $(1, 0)$ over the vector bundle $\Lambda^p_{_{\C}}
T^{\ast}_{X, - J}$,
$$
 \partial_{_{J, p}} : C^{\infty} \left( \Lambda^p_{_{\C}} T^{\ast}_{X, - J}
   \right) \longrightarrow C^{\infty} (\Lambda^{1, 0}_{_J} T^{\ast}_X
   \otimes_{_{\C}} \Lambda^p_{_{\C}} T^{\ast}_{X, - J}) \;, 
$$
given by the formula
$$
 \partial_{_{J, p}} \alpha \hspace{0.25em} (\eta) \;\; : =
   \;\; \eta^{1, 0}_{_J} \;\neg\; \partial_{_J} \alpha \;,
$$
for all $\alpha \in C^{\infty} (X, \Lambda^p_{_{\C}} T^{\ast}_{X, - J})$ and
$\eta \in C^{\infty} (X, T_X \otimes_{_{\R}} \C)$. More explicitly
\begin{eqnarray}
  \partial_{_{J, p}} \alpha \hspace{0.25em} (\eta) \,(\xi_1, ..., \xi_p) & = &
  \eta^{1, 0} . \hspace{0.25em} \alpha (\xi_1, ..., \xi_p)\nonumber
\\\nonumber
  &  & \\
  & + & \sum_{l = 1}^p (- 1)^l\, \alpha ([\eta^{1, 0}, \xi_l^{0, 1}]^{0, 1},
  \xi_1, ..., \hat{\xi}_l, ..., \xi_p) \;,\label{def-del-p}
\end{eqnarray}
for all $\xi_j \in C^{\infty} (X, T_X)$. We observe in particular the identity
$\bar{\partial}_{_{J, p}} \alpha = \overline{\partial_{_{J, p}}
\bar{\alpha}}$. For $p = 1$ the connection $\bar{\partial}_{_{J, 1}}$
writes as
$$
 \bar{\partial}_{_{J, 1}} \alpha (\eta) \cdot \xi \;\; =\;\;
    \eta \hspace{0.25em} . \hspace{0.25em} \alpha (\xi)
   \;\, -\;\, \alpha ([\eta, \xi]^{1, 0})
   \;, 
$$
for all $\eta \in C^{\infty} (X, T^{0, 1}_{X, J})$ and $\xi \in C^{\infty} (X,
T^{1, 0}_{X, J})$. We infer that its dual connection
\[ \bar{\partial}_{_{T^{1, 0}_{X, J}}} : C^{\infty} \left( T^{1, 0}_{X,
   J} \right) \longrightarrow C^{\infty} \left( \Lambda^{0, 1}_{_J} T^{\ast}_X
   \otimes_{_{\C}} T^{1, 0}_{X, J} \right) \hspace{0.25em}, \]
over the bundle $T^{1, 0}_{X, J}$, which is defined by the formula
$$
 ( \bar{\partial}_{_{J, 1}} \alpha) \cdot \xi \hspace{0.75em} =
   \hspace{0.75em} \bar{\partial}_{_J} (\alpha \cdot \xi) 
   \;\,-\;\, \alpha \cdot \bar{\partial}_{_{T^{1, 0}_{X, J}}} \xi
   \;, 
$$
satisfies the identity
$$ \bar{\partial}_{_{T^{1, 0}_{X, J}}} \xi \,(\eta) \;\; =\;\;
    [\eta, \xi]^{1, 0} \; . 
$$
Moreover the canonical $\C$-isomorphism between $T^{1, 0}_{X, J}$ and $T_{X,
J}$ induces a canonical connection
$$
 \bar{\partial}_{_{T_{X, J}}} : C^{\infty} (T_{X, J}) \longrightarrow
   C^{\infty} (T^{\ast}_{X, - J} \otimes_{_{\C}} T_{X, J}) \;, 
$$
of type $(0, 1)$ over the bundle $T_{X, J}$. Explicitly for any $\xi,
\hspace{0.25em} \eta \in C^{\infty} (X, T_X)$,
\begin{eqnarray*}
  \bar{\partial}_{_{T_{X, J}}} \xi (\eta) & : = &
  \bar{\partial}_{_{T^{1, 0}_{X, J}}} \xi^{1, 0} (\eta) \;\, +
  \;\, \overline{\bar{\partial}_{_{T^{1, 0}_{X, J}}} \xi^{1,
  0} (\eta)}\\
  &  & \\
  & = & [\eta^{0, 1} \xi^{1, 0}]^{1, 0} \;\, + \;\,
  [\eta^{1, 0}, \xi^{0, 1}]^{0, 1} \; .
\end{eqnarray*}
Moreover the connection $\bar{\partial}_{_{T_{X, J}}}$ induces in a
natural way a connection
$$
 \partial_{_{T_{X, - J}}} : C^{\infty} (T_{X, - J}) \longrightarrow
   C^{\infty} (T^{\ast}_{X, J} \otimes_{_{\C}} T_{X, - J}) \;, 
$$
of type $(1, 0)$ over the bundle $T_{X, - J}$ by the formula
$$
 \partial_{_{T_{X, - J}}} \xi \,(\eta) \;\; : = \;\;
   \bar{\partial}_{_{T_{X, J}}} \xi \,(\bar{\eta}) \;, 
$$
for all $\xi \in C^{\infty} (X, T_X)$ and $\eta \in T_X \otimes_{_{\R}} \C$.
On the other hand the identity (\ref{def-del-p}) writes for $p = 1$ as
$$
 \partial_{_{J, 1}} \alpha (\eta) \cdot \xi \;\; =
   \;\; \eta \hspace{0.25em} . \hspace{0.25em} \alpha (\xi)
   \;\, - \;\, \alpha ([\eta, \xi]^{0, 1})
   \;, 
$$
for all $\eta \in C^{\infty} (X, T^{1, 0}_{X, J})$ and $\xi \in C^{\infty} (X,
T^{0, 1}_{X, J})$. We infer that its dual connection
$$
 \partial_{_{T^{0, 1}_{X, J}}} : C^{\infty} \left( T^{0, 1}_{X, J} \right)
   \longrightarrow C^{\infty} \left( \Lambda^{1, 0}_{_J} T^{\ast}_X
   \otimes_{_{\C}} T^{0, 1}_{X, J} \right) \;, 
$$
over the bundle $T^{0, 1}_{X, J}$, which is defined by the formula
$$
 (\partial_{_{J, 1}} \alpha) \cdot \xi \;\;= \;\;
   \partial_{_J} (\alpha \cdot \xi) \;\,- \;\, \alpha
   \cdot \partial_{_{T^{0, 1}_{X, J}}} \xi \;, 
$$
satisfies the identity
$$
 \partial_{_{T^{0, 1}_{X, J}}} \xi \,(\eta) \;\;= \;\;
   [\eta, \xi]^{0, 1} \; . 
$$
We deduce that for all $\xi \in C^{\infty} (X, T_X)$ and $\eta \in T_X
\otimes_{_{\R}} \C$ hold the identity
\begin{eqnarray*}
  \bar{\partial}_{_{T_{X, J}}} \xi (\eta) & = &
  \bar{\partial}_{_{T^{1, 0}_{X, J}}} \xi^{1, 0} (\eta) \hspace{0.75em} +
  \hspace{0.75em} \overline{\bar{\partial}_{_{T^{1, 0}_{X, J}}} \xi^{1,
  0} (\eta)}\\
  &  & \\
  & = & \bar{\partial}_{_{T^{1, 0}_{X, J}}} \xi^{1, 0} (\eta)
  \;\,+ \;\,\partial_{_{T^{0, 1}_{X, J}}} \xi^{0, 1} (
  \bar{\eta}) \;.
\end{eqnarray*}
We conclude that the connection $\partial_{_{T_{X, - J}}}$ is the dual of the
connection $\partial_{_{J, 1}}$ over the bundle $T_{X, - J}$.

\subsection{The Levi-Civita connection of a K\"ahler metric}

For convenience the notation for the $\partial$ operator in this subsection
is slightly different from the one used in the paper. Let $(X, J, \omega)$ be
an almost hermitian manifold. Let $g : = \omega (\cdot, J \cdot)$ be the
induced Riemannian metric and $h : = g - i \omega$ be the induced hermitian
metric over $T_{X, J}$. The hermitian data determines two connections of type
$(1, 0)$;
\[ \partial_{_{T_{X, J}}}^{\omega} : = h^{- 1} \cdot \partial_{_{J, 1}} \cdot
   h \hspace{0.25em} : C^{\infty} (T_{X, J}) \longrightarrow C^{\infty}
   (\Lambda^{1, 0}_{_J} T^{\ast}_X \otimes_{_{\C}} T_{X, J}) \hspace{0.25em},
\]
and
\[ \partial_{_{T^{1, 0}_{X, J}}}^{\omega} : = \omega^{- 1} \cdot
   \partial_{_{J, 1}} \cdot \omega \hspace{0.25em} : C^{\infty} (T^{1, 0}_{X,
   J}) \longrightarrow C^{\infty} (\Lambda^{1, 0}_{_J} T^{\ast}_X
   \otimes_{_{\C}} T^{1, 0}_{X, J}) \hspace{0.25em}, \]
where $h$ and $\omega$ are considered as morphisms $h : T_{X, J} \rightarrow
T^{\ast}_{X, - J}$ and $\omega : T_X \rightarrow T^{\ast}_X$. The trivial
identities
\begin{eqnarray*}
 h (\xi, \eta) &= &h (\xi^{1, 0}, \eta) 
\\
\\
&=& h (\xi^{1, 0}, \eta^{0, 1}) 
\\
\\
& =&
    - \;\, 2\, i \, \omega (\xi^{1, 0},
   \eta^{0, 1}) 
\\
\\
& = & - \;\, 2\, i\, \omega (\xi^{1, 0}, \eta)\;, 
\end{eqnarray*}
for all $\xi\,, \, \eta \in C^{\infty} (X, T_X)$
imply
\begin{equation}
\label{key-contract} \xi \;\neg\; h \;=\; -\; 2\; i\, \xi^{1, 0} \;\neg\; \omega\;.
\end{equation}
This combined with the definition of $\partial_{_{T_{X, J}}}^{\omega}$ gives
\begin{eqnarray*}
  \mu \;\;\assign\;\; \partial_{_{T_{X, J}}}^{\omega} \xi \,(\eta) & \assign & h^{- 1}
  \left[ \eta \;\neg\; \partial_{_{J, 1}} (\xi \neg h) \right]\\
  &  & \\
  & = & h^{- 1} \left[ \eta^{1, 0} \;\neg\; \partial_{_J} (\xi \neg h) \right]\\
  &  & \\
  & = & - \;\,2\, h^{- 1} \left[ \eta^{1, 0} \;\neg\; i\, \partial_{_J} (\xi^{1, 0} \;\neg\;
  \omega) \right] \;.
\end{eqnarray*}
If we set
\begin{eqnarray*}
  \alpha \;\; \assign \;\; -\;\, 2 \left[ \eta^{1, 0} \;\neg\; i\, \partial_{_J} (\xi^{1, 0}
  \neg \omega) \right] \;\;=\;\; \mu \;\neg\; h\;,
\end{eqnarray*}
then the identity (\ref{key-contract}) applied to $\mu$ implies
\begin{eqnarray*}
  \mu^{1, 0} \;\; = \;\; \frac{i}{2}\, \omega^{- 1} \alpha \;\;=\;\; \omega^{- 1} \left[
  \eta^{1, 0} \;\neg\; \partial_{_J} (\xi^{1, 0} \;\neg\; \omega) \right] \;.
\end{eqnarray*}
Moreover by definition
\begin{eqnarray*}
  \partial_{_{T^{1, 0}_{X, J}}}^{\omega} \xi^{1, 0} \,(\eta) & : = & \omega^{-
  1} \left[ \eta \;\neg\; \partial_{_{J, 1}} (\xi^{1, 0} \;\neg\; \omega) \right]\\
  &  & \\
  & = & \omega^{- 1} \left[ \eta^{1, 0} \;\neg\; \partial_{_J} (\xi^{1, 0} \;\neg\;
  \omega) \right]\\
  &  & \\
  & = & - \;\, J\,g^{- 1} \left[ \eta^{1, 0} \;\neg\; i\, \partial_{_J}
  (\xi^{1, 0} \;\neg\; g) \right] \; .
\end{eqnarray*}
We infer
$$
 \partial_{_{T_{X, J}}}^{\omega} \xi (\eta) \;\; =\;\;
\partial_{_{T^{1, 0}_{X, J}}}^{\omega} \xi^{1, 0} \,(\eta)
   \;\,+ \;\, \overline{\partial_{_{T^{1, 0}_{X,
   J}}}^{\omega} \xi^{1, 0} \,(\eta)} \; . 
$$
In conclusion we obtain the formula
\begin{equation}
  \label{J-lin-Chern} \partial_{_{T_{X, J}}}^{\omega} \xi \,(\eta) \;\;=\;\; -\;\,
  J\,g^{- 1} \left[ \eta^{1, 0} \;\neg\; i\, \partial_{_J} (\xi^{1, 0}
  \;\neg\; g) \;\,- \;\, \eta^{0, 1} \;\neg\; i\,
  \bar{\partial}_{_J} (\xi^{0, 1} \;\neg\; g) \right] \;.
\end{equation}
The Chern connection
$$
 D_{_{T_{X, J}}}^{\omega} \;\;: =\;\; \partial_{_{T_{X, J}}}^{\omega} \;\,+\;\, \bar{\partial}_{_{T_{X, J}}}\;, 
$$
is $h$-hermitian. From now on we assume that $(X, J, \omega)$ is a K\"ahler
manifold, i.e. $\nabla_g \,J \;=\; 0$. In this case the Levi-Civita connection
$\nabla_g$ coincides with the Chern connection $D_{_{T_{X, J}}}^{\omega}$. We
consider also the components
\begin{eqnarray*}
  \nabla^{1, 0}_{g, J} \hspace{0.25em} \xi & : = & \frac{1}{2}\,  \left(
  \nabla_g \,\xi \;\,- \;\, J\, \nabla_g \,\xi \cdot J
  \right) \;,\\
  &  & \\
  \nabla^{0, 1}_{g, J} \hspace{0.25em} \xi & : = & \frac{1}{2}\,  \left(
  \nabla_g \,\xi \;\,+ \;\, J\, \nabla_g \,\xi \cdot J
  \right) \;,
\end{eqnarray*}
of the complexified Levi-Civita connection. The identities 
\begin{eqnarray*}
\nabla^{1, 0}_{g,
J} &=& \partial_{_{T_{X, J}}}^{\omega} \;,
\\
\\
\nabla^{0, 1}_{g, J} &=&
\bar{\partial}_{_{T_{X, J}}}\;,
\end{eqnarray*}
hold only in restriction to the real
tangent bundle $T_X$. In general hold the identities
\begin{eqnarray*}
  \nabla^{1, 0}_{g, J} \hspace{0.25em} \xi & = & \partial_{_{T^{1, 0}_{X,
  J}}}^{\omega} \xi^{1, 0} \;\,+ \;\,
  \overline{\partial_{_{T^{1, 0}_{X, J}}}^{\omega} \overline{\xi^{0, 1}}}
  \;,\\
  &  & \\
  \nabla^{0, 1}_{g, J} \hspace{0.25em} \xi & = & \bar{\partial}_{_{T^{1,
  0}_{X, J}}} \xi^{1, 0} \;\, + \;\,
  \overline{\bar{\partial}_{_{T^{1, 0}_{X, J}}} \overline{\xi^{0, 1}}}\\
  &  & \\
  & = & \bar{\partial}_{_{T^{1, 0}_{X, J}}} \xi^{1, 0} \;\, +\;\,
   \partial_{_{T^{0, 1}_{X, J}}} \xi^{0, 1} \;.
\end{eqnarray*}
Indeed $\nabla^{1, 0}_{g, J}$ and $\nabla^{0, 1}_{g, J}$ are
respectively the $\C$-linear extensions of $\partial_{_{T_{X, J}}}^{\omega}$
and $\bar{\partial}_{_{T_{X, J}}}$ to the complexified tangent bundle $T_X
\otimes_{_{\R}} \C$. In particular $\nabla^{0, 1}_{g, J}$ is independent of the metric
$g$.
Let now $(\zeta_k)_{k = 1}^n \subset \mathcal{O}(U, T^{1, 0}_{X, J})$ be a
local holomorphic frame and write $\xi \;=\; \xi'_k\, \zeta_k \;+\; \xi''_k\, 
\bar{\zeta}_k$. Then hold the local expressions
\begin{eqnarray*}
  \nabla^{1, 0}_{g, J} \hspace{0.25em} \xi & = & \left( \zeta_p
  \hspace{0.25em} . \hspace{0.25em} \xi'_k \;\, + \;\,
  A^p_{k, l} \, \xi'_l \right) \zeta^{\ast}_p \otimes \zeta_k
  \;\, + \;\, \left( \bar{\zeta}_p \hspace{0.25em} .
  \hspace{0.25em} \xi''_k \;\,+ \;\, \overline{A^p_{k,
  l}} \, \xi''_l \right) \bar{\zeta}^{\ast}_p \otimes
  \bar{\zeta}_k \;,\\
  &  & \\
  \nabla^{0, 1}_{g, J} \hspace{0.25em} \xi & = & ( \bar{\zeta}_p
  \hspace{0.25em} . \hspace{0.25em} \xi'_k) \hspace{0.25em}
  \bar{\zeta}^{\ast}_p \otimes \zeta_k \;\,+ \;\,
  (\zeta_p \hspace{0.25em} . \hspace{0.25em} \xi''_k) \hspace{0.25em}
  \zeta^{\ast}_p \otimes \bar{\zeta}_k \;,
\end{eqnarray*}
where $A_{k, l}^p : = (\zeta_p \hspace{0.25em} . \hspace{0.25em} \omega_{l,
\bar{r}}) \,\omega^{r, \bar{k}}$ and 
$$
\omega \;\;=\;\; \frac{i}{2} \;\omega_{k, \bar{l}} \;\zeta_k^{\ast}
\wedge \bar{\zeta}^{\ast}_l\;.
$$ 
In particular we find the following
expressions;
\begin{equation}
  \label{cov-der-Tg} \nabla_g \hspace{0.25em} \zeta_l \;\; =
  \;\; A^p_{k, l} \hspace{0.25em} \zeta^{\ast}_p \otimes \zeta_k
  \hspace{0.25em}, \hspace{2em} \nabla_g \hspace{0.25em} \bar{\zeta}_l
  \;\; = \;\; \overline{A^p_{k, l}} \hspace{0.25em} 
  \bar{\zeta}^{\ast}_p \otimes \bar{\zeta}_k \;.
\end{equation}
We infer the following identities for the dual connection on the complexified
cotangent bundle $T^{\ast}_{_X} \otimes_{_{\R}} \C$;
\begin{equation}
  \label{cov-der-cTg} \nabla_g \hspace{0.25em} \zeta^{\ast}_l 
\;\;=\;\;  - \;\, A^p_{l, k} \hspace{0.25em}
  \zeta^{\ast}_p \otimes \zeta^{\ast}_k \hspace{0.25em}, \hspace{2em} \nabla_g
  \hspace{0.25em} \bar{\zeta}^{\ast}_l \;\; = \;\; -
  \;\, \overline{A^p_{l, k}} \hspace{0.25em}  \bar{\zeta}^{\ast}_p
  \otimes \bar{\zeta}^{\ast}_k \;.
\end{equation}
\subsection{Complex operators acting on alternating tensors}\label{cpx-op}

Let $(F, h)$ be a hermitian vector bundle over a K\"ahler manifold $(M,J,
g)$ (of complex dimension $n$) equipped with a $h$-hermitian connection $\nabla_F$ and
let $\nabla$ be the induced hermitian connection over the hermitian vector
bundle
\[ \left( (T^{\ast}_M)^{\otimes p} \otimes_{_{\mathbbm{R}}} F, \left\langle
   \cdot, \cdot \right\rangle \right) \hspace{0.25em}, \]
where $\left\langle \cdot, \cdot \right\rangle$ is the induced hermitian
product. We define the complex operators
\begin{eqnarray*}
  2\, \partial_F & \assign & \nabla_F \;\,-\;\, J_{_F} \nabla_{F, J \cdot}\;,\\
  &  & \\
  2\, \bar{\partial}_F & \assign & \nabla_F \;\,+\;\, J_{_F} \nabla_{F, J \cdot}\;,
\end{eqnarray*}
acting on sections of $F$. We will denote with the same notations their
extension to the sheaf of $F$-valued differential forms. We notice the decomposition formula 
$\nabla_F=\partial_F+\bar{\partial}_F$. We define also the complex
operators
\begin{eqnarray*}
  2\, \nabla_{_J}^{1, 0} & \assign & \nabla \;\,-\;\, J_{_F} \nabla_{_{J \cdot}}\;,\\
  &  & \\
  2\, \nabla_{_J}^{0, 1} & \assign & \nabla \;\,+\;\, J_{_F} \nabla_{_{J \cdot}}\;,
\end{eqnarray*}
acting on $F$-valued tensors. Let now $(e_k)_k$ be a $J$-complex and $g
J$-orthonormal basis of $T_M $. We expand the formula (see the appendix in \cite{Pal1})
\begin{eqnarray*}
  \nabla_F^{\ast} \,\alpha & = & - \;\,p \tmop{Tr}_g \nabla \alpha\\
  &  & \\
  & = & - \;\,p\, \nabla \alpha \,(e_j, e_j, \cdot) \;\,-\;\, p\, \nabla \alpha \,(J e_j, J e_j,
  \cdot)\\
  &  & \\
  & = & - \;\,2\, p \,\nabla \alpha \,(\zeta_j, \bar{\zeta}_j, \cdot) \;\,-\;\, 2\, p\, \nabla
  \alpha \,( \bar{\zeta}_j, \zeta_j, \cdot) \;.
\end{eqnarray*}
The fact that the metric $g$ is K\"{a}hler implies that the operator $\nabla_{\bullet}\alpha$ preserves the be-degrees. Thus by be-degree reasons we infer the formulas
\begin{eqnarray}
  \partial_F^{\ast} \,\alpha & = & - \;\,2\, p\, \nabla \alpha \,( \bar{\zeta}_j, \zeta_j,
  \cdot) \;\;=\;\; -\;\, p \tmop{Tr}_g \nabla_{_J}^{0, 1} \alpha\;,\label{expr-adDel}
\\\nonumber
  &  & \\
  \bar{\partial}_F^{\ast} \,\alpha & = & -\;\, 2\, p\, \nabla \alpha \,(\zeta_j,
  \bar{\zeta}_j, \cdot) \;\;=\;\; -\;\, p \tmop{Tr}_g \nabla_{_J}^{1, 0} \alpha \;.\label{expr-adDbar}
\end{eqnarray}
We observe also the formulas
\begin{eqnarray}
  \partial_F \, \alpha \,(\xi_0, ..., \xi_p) & = & \sum_{j = 0}^p
  (- 1)^j  \hspace{0.25em} \nabla_{_J}^{1, 0} \alpha \,(\xi_j, \xi_0, ...,
  \hat{\xi}_j, ..., \xi_p) \;,\label{expr-del}
\\\nonumber
  &  & \\
  \bar{\partial}_F \,\alpha \,(\xi_0, ..., \xi_p) & = & \sum_{j = 0}^p (-
  1)^j \hspace{0.25em} \nabla_{_J}^{0, 1} \alpha \,(\xi_j, \xi_0, ...,
  \hat{\xi}_j, ..., \xi_p) \; .\label{expr-dbar}
\end{eqnarray}
It is sufficient to show the identities (\ref{expr-del}), (\ref{expr-dbar}) for a real basis $(\xi_j)$ of
$T_M$. So let $(z_1, \ldots, z_n)$ be $J$-holomorphic and $g$-geodesic
coordinates centered at a point $p_0$ and let $\xi_j \assign
\frac{\partial}{\partial x_j}$ or $\xi_j \assign
\frac{\partial}{\partial y_j}$, where $z_j \;=\; x_j \;+\; i\, y_j$. We infer
\begin{eqnarray*}
  \partial_F \, \alpha\, (\xi_0, ..., \xi_p) & = & \sum_{j = 0}^p
  (- 1)^j  \hspace{0.25em} \partial_{F, \xi_j} \left[ \alpha (\xi_0, ...,
  \hat{\xi}_j, ..., \xi_p) \right] \\
  &  & \\
  & = & \frac{1}{2} \, \sum_{j = 0}^p (- 1)^j   \left(
  \nabla_{F, \xi_j} \;\,-\;\, J_{_F} \nabla_{F, J \xi_j} \right) 
  \left[ \alpha (\xi_0, ..., \hat{\xi}_j, ..., \xi_p) \right]\;,
\end{eqnarray*}
and thus at the point $p_0$ hold the equalities
\begin{eqnarray*}
  \partial_F \,\alpha\, (\xi_0, ..., \xi_p) & = & \frac{1}{2} \,
  \sum_{j = 0}^p (- 1)^j \left( \nabla_{\xi_j} \alpha \;\,-\;\,
  J_{_F} \nabla_{J \xi_j} \alpha \right) (\xi_0, ..., \hat{\xi}_j, ..., \xi_p)
  \hspace{0.25em} \\
  &  & \\
  & = & \sum_{j = 0}^p (- 1)^j \hspace{0.25em} \nabla_{_J}^{1, 0} \alpha
  \,(\xi_j, \xi_0, ..., \hat{\xi}_j, ..., \xi_p) \;,
\end{eqnarray*}
since $\nabla_g \,\xi_j (p_0) = 0$. The proof of the identity (\ref{expr-dbar}) is quite similar.

\vspace{1cm}
\noindent
Nefton Pali
\\
Universit\'{e} Paris Sud, D\'epartement de Math\'ematiques 
\\
B\^{a}timent 425 F91405 Orsay, France
\\
E-mail: \textit{nefton.pali@math.u-psud.fr}

\end{document}